\documentclass[11pt,reqno]{amsart}

\usepackage{amsmath,amssymb,amsthm, amsfonts, mathrsfs}
\usepackage{hyperref}
\usepackage{xcolor}
\usepackage{textcomp}
\usepackage{enumerate}
\usepackage{amsmath}
\usepackage{calligra,indentfirst,epsfig}
\usepackage{MnSymbol}
\calligra

\usepackage[left=3cm,right=3cm,top=3cm,bottom=3cm]{geometry}
\usepackage{lipsum}
\usepackage{ytableau}
\usepackage{tikz}
\usepackage{xcolor}

\newtheorem{thm}{Theorem}[section]
\newtheorem{lemma}[thm]{Lemma}
\newtheorem{proposition}[thm]{Proposition}
\newtheorem{corollary}[thm]{Corollary}
\newtheorem{conjecture}{Conjecture}[section]

\theoremstyle{definition}
\newtheorem{definition}{Definition}[section]

\newtheorem{remark}{Remark}[section]
\newtheorem{example}{Example}[section]

\allowdisplaybreaks
\usepackage{etoolbox}
\usepackage[backend=bibtex, style=numeric]{biblatex}
\AtBeginBibliography{\pagestyle{plain}} 
\newcommand{\pn}{\mathcal{P}[n]}
\newcommand{\complex}{\mathbb{C}}
\newcommand{\reals}{\mathbb{R}}

\newcommand{\integers}{\mathbb{Z}}
\newcommand{\zplus}{\integers_+}
\newcommand{\lie}{\mathfrak}

\newcommand{\wnot}{w_0}
\newcommand{\kkmod}[1][w]{K(\lambda,#1,\mu)}
\newcommand{\mct}{\mathcal{T}}

\newcommand{\lr}{c_{\lambda\mu}^{\,\nu}}
\newcommand{\lrw}[1][w]{c_{\lambda\mu}^{\,\nu}(#1)}
\newcommand{\lrevw}[1][w]{c_{\mu\lambda}^{\,\nu}(#1)}
\newcommand{\lrwk}[1][w]{c_{k\lambda,k\mu}^{\,k\nu}(#1)}
\newcommand{\lrwd}[1][w]{c_{d\lambda,d\mu}^{\,d\nu}(#1)}
\newcommand{\supstd}[1][]{b_{#1}}
\newcommand{\supstdual}[1][]{b^{*}_{#1}}

\newcommand{\dchar}[1][\mu]{\chi_w({#1})}
\newcommand{\bighive}{\mathbf{\Delta}}
\newcommand{\ps}[1][]{\overline{#1}}
\newcommand{\lmn}{\lambda, \mu, \nu}
\newcommand{\mln}{\mu, \lambda, \nu}
\newcommand{\sd}{\partial}
\newcommand{\bsd}{{\boldsymbol \partial}}
\renewcommand{\iff}{\Leftrightarrow}
\newcommand{\kcone}[1][\sigma]{\koganH(\text{ --} \,, #1)}
\newcommand{\kzcone}[1][\sigma]{\koganHz(\text{ --} \,, #1)}
\newcommand{\rootthree}{1.7320508}
\newcommand{\id}{\mathbf{1}}

\DeclareMathOperator{\SE}{SE}
\DeclareMathOperator{\NE}{NE}

\DeclareMathOperator{\crys}{\mathcal{B}}
\DeclareMathOperator{\demcrys}{Tab}

\newcommand{\dem}{T}
\DeclareMathOperator{\sgn}{sgn}
\DeclareMathOperator{\U}{\mathfrak{U}}
\DeclareMathOperator{\ch}{char}
\DeclareMathOperator{\Tab}{Tab}
\DeclareMathOperator{\Hive}{\mathbb{H}ive}
\DeclareMathOperator{\Hivez}{\Hive_\integers}
\DeclareMathOperator{\GT}{GT}
\DeclareMathOperator{\GTz}{GT_\integers}
\DeclareMathOperator{\kogan}{\GT}
\DeclareMathOperator{\koganz}{\GT_\integers}
\DeclareMathOperator{\koganH}{\Hive}
\DeclareMathOperator{\koganHz}{\Hive_\integers}
\DeclareMathOperator{\bg}{\mathbf{\Gamma}}
\DeclareMathOperator{\len}{len}

\DeclareMathOperator{\horn}{\mathbb{H}orn}
\DeclareMathOperator{\supp}{supp}
\DeclareMathOperator{\wt}{wt}
\DeclareMathOperator{\res}{res}

\begin{document} 
 
\title[Saturation for refined LR coefficients]{ The Saturation property for refined Littlewood--Richardson coefficients}

\author[]{Mrigendra Singh Kushwaha}
\address{Department of Mathematics, Faculty of Mathematical Science, University of Delhi, New Delhi - 110007}
\email{mrigendra154@gmail.com, mskushwaha@maths.du.ac.in}

\author[]{K. N. Raghavan}
\address{School of Interwoven Arts and Sciences, Krea University, Sri City, India}
\email{raghavan.komaranapuram@krea.edu.in}

\author[]{Sankaran Viswanath}
\address{The Institute of Mathematical Sciences, A CI of Homi Bhabha National Institute, Chennai, India}
\email{svis@imsc.res.in}

\thanks{MSK acknowledges financial support from a CV Raman postdoctoral fellowship at IISc, Bengaluru. SV acknowledges partial financial support under SERB MATRICS grant number MTR/2019/000071.}
\subjclass[2020]{17B10, 05E10}
\keywords{Kostant-Kumar modules, refined Littlewood--Richardson coefficients, hive model, pattern-avoiding permutations, saturation property, semigroup property.}
\date{}
 
\begin{abstract} 
Given dominant integral weights $\lmn$ of a finite-dimensional simple Lie algebra $\lie g$ and an element $w$ of its Weyl group, the refined tensor product multiplicity $c_{\lambda \mu}^\nu(w)$ is the multiplicity of the irreducible $\lie g$-module $V(\nu)$ in the so-called Kostant--Kumar submodule $K(\lambda, w, \mu)$ of the tensor product $V(\lambda) \otimes V(\mu)$. We derive properties of these coefficients in general type, including a Brauer--Klimyk type formula and restriction theorems. In type $A$, we obtain a hive model for the $\lrw$  and prove that the saturation and strong semigroup properties hold if the permutation $w$ is  $312$-avoiding, $231$-avoiding, or a commuting product of such elements. This generalizes the classical Knutson--Tao saturation theorem.

 \end{abstract}
 
\maketitle 

\section{Introduction}
Let $\lie g$ be a finite-dimensional simple Lie algebra over $\complex$ and let $V(\lambda)$ denote its irreducible finite dimensional representation, indexed by the  dominant integral weight $\lambda$. The extremal weight vectors of $V(\lambda)$ are the nonzero weight vectors whose weights are Weyl conjugates of $\lambda$. Let $W$ be the Weyl group and $W_\lambda$ the stabilizer of $\lambda$ in $W$.

Given $\lambda, \mu$, the tensor product $V(\lambda) \otimes V(\mu)$ has a family of submodules $\kkmod$ indexed by elements $w$ of the double coset space $W_\lambda \backslash W / W_\mu$. These are called {\em Kostant--Kumar modules}, and are cyclic modules generated by the tensor product of pairs of extremal vectors of $V(\lambda)$ and $V(\mu)$ \cite{KRV}.

 The main objective of this paper is to study the multiplicities $\lrw$ obtained when $\kkmod$ is decomposed into irreducibles:
 \[K(\lambda,w,\mu) = \bigoplus_{\nu \in P^+} V(\nu)^{\oplus \lrw}. \]
If $w_0$ denotes the longest element of the Weyl group, we have  $\kkmod[w_0] = V(\lambda) \otimes V(\mu)$. We may thus view $\lrw$ as a $w$-refined version of the tensor product multiplicity $\lr$, with $\lrw[w_0]$ coinciding with $\lr$.
Kostant--Kumar modules form an increasing filtration of the tensor product and for fixed $(\lmn)$, the multiplicity $\lrw$ increases as $w$ increases in the Bruhat order on double cosets. We refer the reader to Section~\ref{sec:kkmodules} for an explanation of undefined terms.

We first obtain an alternating sum formula for the $\lrw$ in all types which may be viewed as an analogue of the classical Brauer--Klimyk formula for the $\lr$. We then study the saturation and semigroup properties of the $\lrw$ in type $A$.
Our main theorem in this setting asserts that $w \in W\cong S_n$ has the saturation and strong semigroup properties if $w$ (when viewed as a permutation) avoids the patterns $312$ or $231$, or is a commuting product of such permutations in a Young subgroup.
Permutations of the special form above also appear in work of Postnikov--Stanley, where explicit formulas for the degree polynomials of the corresponding Schubert varieties were established \cite{PS}.

While our theorem generalizes Knutson--Tao's saturation theorem, our proof makes essential use of their {\em hive model} for the Littlewood--Richardson coefficients. 
Along the way, we derive a polyhedral description of the $\lrw$  in terms of the {\em Kogan faces} of the hive polytope.

It is worth mentioning here that there are now several different proofs of the Knutson--Tao saturation theorem in type $A$; for example Derksen--Weyman's proof via quiver varieties \cite{derksen-weyman} and Kapovich--Millson's proof via geodesic triangles in buildings \cite{kapovich-millson} among others. While Knutson--Tao's proof technique via hives is naturally amenable to $w$-refinement (at least, for the special $w$ we consider in this paper), it is much less obvious how the Weyl group element can be injected into the arguments of these other proofs.

The paper is organized as follows. Section~\ref{sec:kkmodules} reviews the definition of Kostant--Kumar modules. The main results of this section are a Brauer--Klimyk type formula for the refined tensor product multiplicities $\lrw$ (Proposition~\ref{prop:strongcor}) and the restriction theorems (Theorem~\ref{thm:lrwres} and Proposition~\ref{prop:w1w2mult}). 
Section~\ref{sec:semigpsatdef} defines the semigroup and saturation properties for general elements of the Weyl group and proposes a conjecture on existence of saturation factors (Conjecture~\ref{conj:sat-fact-exists}). Sections \ref{sec:demcrys}--\ref{sec:right-keys} are concerned exclusively with the type $A$ case. Section \ref{sec:demcrys} recalls a result of Fujita on Demazure crystals in the Gelfand-Tsetlin model. We employ this in Section~\ref{sec:hive-clmn} to derive a polyhedral  expression for $\lrw$ in terms of the {\rm Kogan faces} of the hive polytope (Theorem~\ref{thm:hive-desc-lrw}). Section~\ref{sec:patt312} states the key result of this paper (Theorem~\ref{thm:mainthm-312}) on the saturation and strong semigroup properties for certain classes of pattern-avoiding permutations, and uses results of Postnikov-Stanley \cite{PS} to derive yet another formula (Theorem~\ref{thm:postan312}) for the $\lrw$ when $w$ is of special form.  Sections~\ref{sec:inchives}, \ref{sec:hivehorncones} prove our saturation theorem via a modification of techniques of Knutson-Tao. Section~\ref{sec:strongthm} proves an extension of the saturation theorem to all permutations, at the expense of imposing restrictions on the triples $(\lambda,\mu,\nu)$. Finally, Section~\ref{sec:right-keys} provides a bijective proof of the equality $\lrw = c_{\mu\lambda}^\nu(w^{-1})$ in type $A$ via the hive model.

Some of the results of this paper were announced as part of an extended abstract in FPSAC 2021 \cite{KRV_Sat}.
\section{Kostant--Kumar modules and refined multiplicities in tensor products}\label{sec:kkmodules}

\subsection{}
Unless otherwise mentioned, we will let $\lie g$ denote a finite-dimensional simple Lie algebra over $\complex$. Let $\lie b$ be a fixed choice of Borel subalgebra. We let $ P, P^+, \Delta^+$ respectively denote the weight lattice and the sets of dominant integral weights and positive roots of $\lie g$. Let $W=W(\lie g)$ be the Weyl group of $\lie g$, $Q$ its root lattice and $Q^+ = \integers_{\geq 0}(\Delta^+)$.
Let $V(\lambda)$ denote the finite-dimensional irreducible representation of $\lie g$ indexed by its highest weight $\lambda \in P^+$.

Given $\lambda, \mu \in P^+$ and $\sigma \in W$, let $v_{\sigma\lambda}$, $v'_{\sigma\mu}$  denote nonzero vectors of weight $\sigma\lambda$ and $\sigma\mu$  in $V(\lambda)$ and $V(\mu)$ respectively. Let $\U \lie g$ be the universal enveloping algebra of $\lie g$.
Given $w \in W$, the {\em Kostant--Kumar module} $K(\lambda, w, \mu)$ is the cyclic $\lie g$-submodule of the tensor product $V(\lambda) \otimes V(\mu)$ generated by $v_\lambda \otimes v'_{w\mu}$ \cite{Kumar,KRV}.
\begin{equation}\label{eq:kkmoddef}
K(\lambda, w, \mu) = \U\lie g (v_\lambda \otimes v'_{w\mu}) \subseteq V(\lambda) \otimes V(\mu) 
\end{equation}

We now recall the following properties of Kostant--Kumar modules from \cite{KRV}.
For $\lambda \in P^+$, let $W_\lambda$ denote the stabilizer of $\lambda$ in $W$.
\begin{proposition}\label{prop:kkmod}
(\cite{KRV}) Let $\sigma, \tau \in W$ and  $\lambda, \mu \in P^+$.
\begin{enumerate}
\item $K(\lambda, w, \mu) = \U \lie g(v_{\sigma \lambda} \otimes v'_{\sigma w\mu})$ for all $\sigma \in W$.
\item  $K(\lambda,\sigma,\mu) = K(\lambda,\tau,\mu)$ if $W_\lambda \sigma W_\mu = W_\lambda \tau W_\mu$.
\item $K(\lambda,\sigma,\mu) \subseteq K(\lambda,\tau,\mu)$  if $W_\lambda \sigma W_\mu \leq W_\lambda \tau W_\mu$ in the Bruhat poset  $W_\lambda \backslash W /W_\mu$ of double cosets.
\item $K(\lambda,\sigma,\mu) \cong K(\mu,\sigma^{-1},\lambda)$. More precisely, $K(\lambda,\sigma,\mu)$ maps to $K(\mu,\sigma^{-1},\lambda)$ under the $\lie g$-isomorphism  $V(\lambda)\otimes V(\mu) \to V(\mu) \otimes V(\lambda)$ mapping $v_1 \otimes v_2 \mapsto v_2 \otimes v_1$.
\item $K(\lambda,\id,\mu) \cong V(\lambda+\mu)$ and $K(\lambda,w_0,\mu) = V(\lambda) \otimes V(\mu) $, where $\id$ denotes the identity and $w_0$ the longest element in $W$.
\end{enumerate}
\end{proposition}

We decompose $K(\lambda,w,\mu)$ into irreducibles:
\begin{equation}\label{eq:repthy-lrw} 
K(\lambda,w,\mu) = \bigoplus_{\nu \in P^+} V(\nu)^{\oplus \lrw}. 
\end{equation}
The multiplicities $\lrw$ so obtained are termed the $w$-refined tensor product multiplicities (or $w$-refined Littlewood--Richardson coefficients in type $A$).
In light of Proposition~\ref{prop:kkmod}, we obtain the decomposition of the full tensor product by setting  $w=w_0$:
\[ V(\lambda) \otimes V(\mu) = \bigoplus_{\nu \in P^+} V(\nu)^{\oplus \lr} \]
where $\lr = \lrw[w_0]$ are the tensor product multiplicities. When we need to emphasize the ambient Lie algebra $\lie g$, we will write $K(\lambda, w, \mu; \, \lie g)$ and $\lrw[w; \lie g]$ in place of $K(\lambda, w, \mu)$ and $\lrw$.

\subsection{} Proposition~\ref{prop:kkmod} directly implies the following key properties of the $\lrw$.
\begin{proposition}\label{prop:properties-clmn} Let $\lambda, \mu, \nu \in P^+$. Then
  \begin{enumerate}
  \item[(a)] $\lrw \in \zplus$ for all $w \in W$.
  \item[(b)] $\lrw[\sigma] = \lrw[\tau]$ if $W_\lambda \sigma W_\mu = W_\lambda \tau W_\mu$.
  \item[(c)] $\lrw[\sigma] \leq \lrw[\tau]$ if $W_\lambda \sigma W_\mu \leq W_\lambda \tau W_\mu$.
  \item[(d)] $\lrw = c_{\mu \lambda}^\nu(w^{-1})$ for all $w \in W$.
  \item[(e)] $\lrw[\wnot] = \lr$ and $\lrw[\id] = \delta_{\lambda+\mu, \nu}$.
   \end{enumerate}
\end{proposition}
We recall that if $\sigma, \tau \in W$ with $\sigma \leq \tau$ in the Bruhat order, then 
$W_\lambda \sigma W_\mu \leq W_\lambda \tau W_\mu$ in the Bruhat order on the double coset space \cite{KRV}.
Thus, Proposition~\ref{prop:properties-clmn}(c) shows that for all $\lambda, \mu, \nu \in P^+$, the map $w \mapsto \lrw$ is an increasing function of posets $W \to \zplus$. Figure~\ref{fig:wrefined-values} shows  an example for $\lie g = \mathfrak{sl}_4$, with the values $\lrw$ superimposed on the Bruhat graph of $W=S_4$.
A bijective proof of Proposition~\ref{prop:properties-clmn}(d) for $\lie g = \mathfrak{sl}_n$ using the hive model occurs below in \S\ref{sec:right-keys}.

\begin{figure}
\centering
\includegraphics[width=0.5\textwidth,angle=270]{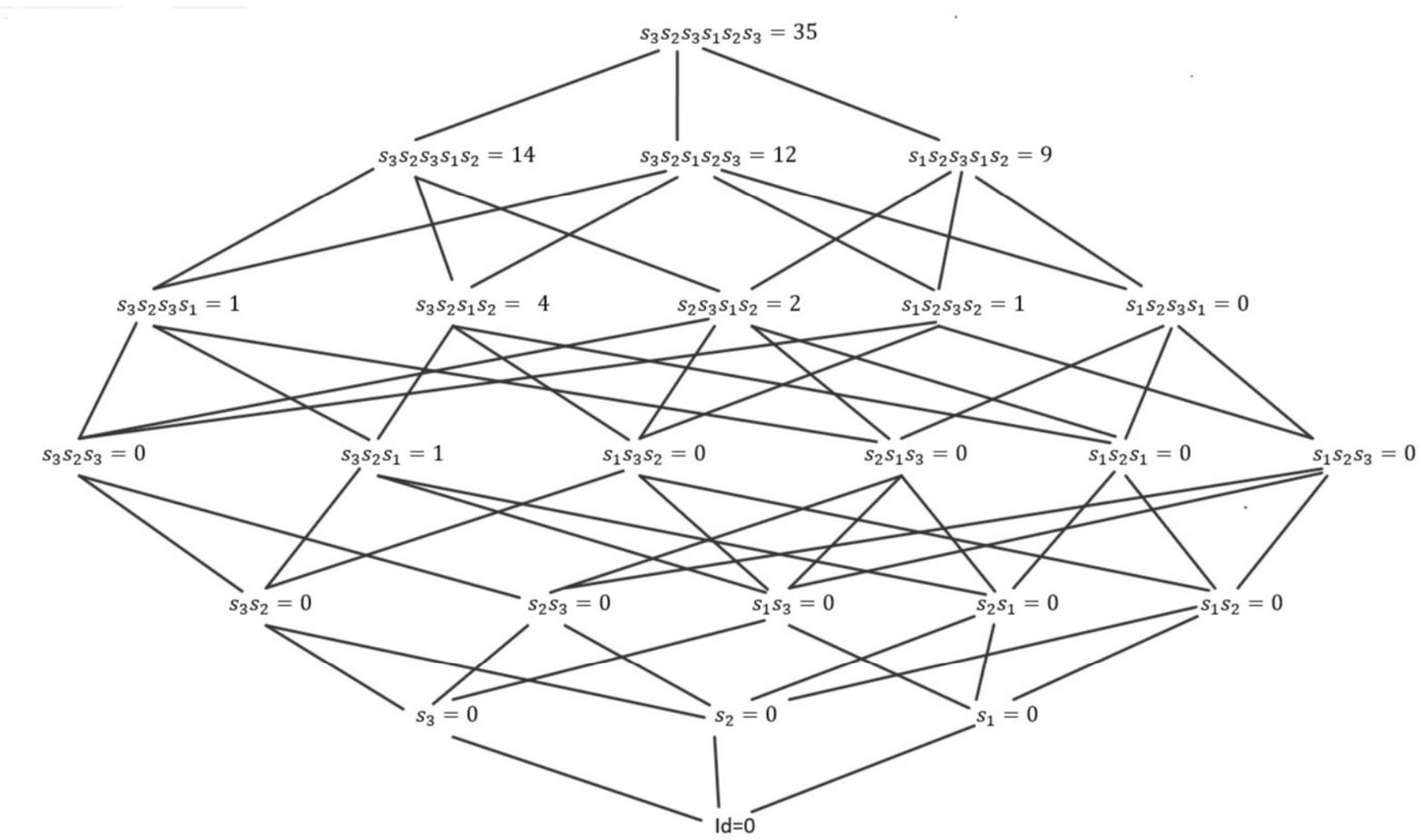}
\caption{Values of $\lrw$ (superimposed on the Bruhat graph of $S_4$) for   $\lambda = (13, 7, 4, 0), \;\; \mu = (13, 7, 2, 0), \;\; \nu = (21, 12, 9, 4)$.}

\label{fig:wrefined-values}
\end{figure}

\subsection{}
The notion of Kostant--Kumar modules has a straightforward extension to the setting of $\lie g$ semisimple (rather than simple). It is defined exactly as before, and is given by Equation~\eqref{eq:kkmoddef}. It has an alternate description in terms of the simple components of $\lie g$, and we record this for future use below.

Writing $\lie g = \oplus_{i=1}^r \lie g_i$ with $\lie g_i$ simple, the Cartan subalgebra $\lie h = \oplus_{i=1}^r \lie h_i$.
Given dominant integral weights $\lambda, \mu$ of $\lie g$, define $\lambda_i, \mu_i$ to be the restrictions of $\lambda, \mu$ to $\lie h_i$  for $i=1, \ldots, r$. Thus,  $\lambda = \sum_{i=1}^r \lambda_i$ and $\mu = \sum_{i=1}^r \mu_i$.  The corresponding irreducible representations of $\lie g$ are isomorphic to the external tensor products 
\[ V(\lambda) \cong \displaystyle\boxtimes_{i=1}^r V(\lambda_i), \;\;\; V(\mu) \cong \boxtimes_{i=1}^r V(\mu_i).\]
Given $w = \prod_{i=1}^r w_i \in W(\lie g)$ with $w_i \in W(\lie g_i)$, it is easy to see that the Kostant--Kumar module $K(\lambda, w, \mu)$ of $\lie g$ is isomorphic to the external tensor product of Kostant--Kumar modules for the $\lie g_i$:
\begin{equation}\label{eq:kktensor}
    K(\lambda, w, \mu) \cong \boxtimes_{i=1}^r K(\lambda_i, w_i, \mu_i; \, \lie g_i)
\end{equation}
Given a dominant integral weight $\nu = \sum_{i=1}^r \nu_i$ of $\lie g$ with $\nu_i:=\nu|_{\lie h_i}$,  we obtain from \eqref{eq:kktensor} that:
\begin{equation}\label{eq:cprod}
    \lrw = \dprod_{i=1}^r c_{\lambda_i, \mu_i}^{\nu_i}(w_i; \, \lie g_i)
\end{equation}
 
\subsection{} We return to our standing assumption that $\lie g$ is simple. Let $S=\{1, 2, \ldots, n\}$ index the nodes of the Dynkin diagram of $\lie g$. For $i \in S$, let $\alpha_i$ denote the corresponding simple root, $\alpha_i^\vee$ the simple coroot, $e_i, f_i$ the Chevalley generators and  $s_i \in W$ the simple reflection. 
Let $\rho=\sum_{\alpha \in \Delta^+} \alpha/2$ denote the Weyl vector.

Given $w \in W$, fix a reduced expression for $w$ and let $\supp(w)$ denote the set of all $i \in S$ such that $s_i$ occurs in the chosen reduced expression of $w$ (this is independent of the choice of reduced expression). 

Likewise, given an element $\beta = \sum_{i \in S} c_i \alpha_i \in \lie h^*$, let $\supp(\beta)$ denote the set of $i \in S$ such that $c_i \neq 0$.   
Let  $Q_I:=\{\alpha \in Q: \supp(\alpha) \subset I\}$ and $Q^+_I = Q_I \cap Q^+$.

We recall the following simple fact:
\begin{lemma}\label{lem:simplem}
Let $\mu \in P^+$ and let $w \in W$ with $I=\supp(w)$. Then $\mu-w\mu \in Q^+_I$. Further, if $\mu$ is regular dominant, i.e., $\mu \in (\rho+P^+)$, then $\supp(\mu-w\mu) = I$.
\end{lemma}

\subsection{}
Given $w \in W$ and $\mu \in P^+$, the {\em Demazure module} $V_w(\mu)$ is the cyclic $\U\lie b$-submodule of $V(\mu)$ generated by a nonzero vector $v'_{w\mu}$ of weight $w\mu$. This section collects together well-known facts about Demazure modules; brief proofs are indicated where appropriate, since there is no single convenient reference 
we can point the reader to for the facts stated below.

The character $\dchar$ of $V_w(\mu)$ is defined by 
\[\dchar = \sum_{\alpha \in Q} d_{w,\mu}(\alpha) \,e^{\mu-\alpha}\] 
where $d_{w,\mu}(\alpha) := \dim V_w(\mu)_{\mu-\alpha}$ is the dimension of the 
$(\mu-\alpha)$-weight space of $V_w(\mu)$. If we wish to emphasize the role of the ambient algebra $\lie g$, we will write $d_{w,\mu}(\alpha; \,\lie g)$ and $\dchar[\mu; \, \lie g]$ in place of $d_{w,\mu}(\alpha)$ and $\dchar$.
Since every weight $\gamma$ of $V_w(\mu)$ satisfies $w\mu \leq \gamma \leq \mu$, Lemma~\ref{lem:simplem} implies the following statement:
\begin{lemma}\label{lem:demwt} Let $\mu \in P^+$ and $w \in W$ with $\supp(w) = I$. If
   $d_{w,\mu}(\alpha) > 0$, then $\alpha \in Q^+_I$. 
\end{lemma}
The Demazure character formula states that  
\begin{equation}\label{eq:demcharformula} 
\dchar = \dem_w(e^\mu) 
\end{equation} where $T_w: \complex[P] \to \complex[P]$ is the Demazure operator, defined as follows. For a simple reflection $s_i$ and $\gamma \in P$, let 
\begin{equation}\label{eq:twdef}
\dem_i(e^\gamma) = \frac{e^\gamma - e^{s_i(\gamma)-\alpha_i}}{1 - e^{-\alpha_i}} \end{equation}
If $w= s_{i_1} s_{i_2} \cdots s_{i_k}$ is a reduced word for $w \in W$, define $T_w := T_{i_1} T_{i_2} \cdots T_{i_k}$. This is independent of the choice of reduced word.

\subsection{}
Given $I \subset S$, let $\lie h_I$ denote the span of the coroots $\{\alpha_i^\vee: i \in I\}$. For $\mu \in \lie h^*$, let $\mu_I \in \lie h_I^*$ denote its restriction to $\lie h_I$. 
Let ${\lie g}_I$ denote the (semisimple) Lie subalgebra of $\lie g$ generated by $\lie h_I$ and $e_i, f_i$ for $i \in I$. Let $W_I:=\{w \in W: \supp(w) \subset I\}$ denote the parabolic subgroup of $W$ generated by $s_i$, $i \in I$. We will identify $W_I$ with the Weyl group of $\lie g_I$. Likewise, we identify the weight lattice of $\lie g_I$ with $P_I:=\{\mu_I: \mu \in P\}$.

The simple coroots of $\lie g_I$ are $\alpha_i^\vee, \; i \in I$. Let $\alpha_i^I \in \lie h_I^*, \, i \in I$ denote the simple roots of $\lie g_I$. Since $\langle \alpha_i, \alpha_j^\vee\rangle = \langle (\alpha_i)_I, \alpha_j^\vee \rangle$ for all $i, j \in I$, we see that $\alpha_i^I$ coincides with the restriction $(\alpha_i)_I$ of $\alpha_i$ to $\lie h_I$. Thus, the restriction map $\alpha \mapsto \alpha_I$ defines an isomorphism from $Q_I=\{\alpha \in Q: \supp(\alpha) \subset I\}$ to the root lattice $Q(\lie g_I)$ of $\lie g_I$. We will identify $Q_I$ and $Q(\lie g_I)$ via this map.

\begin{lemma}\label{lem:sres}
    For $\mu \in P$, $(s_{\alpha_i}(\mu))_I = s_{\alpha_i^I}(\mu_I)$ for all $i \in I$. More generally, $(w\mu)_I = w(\mu_I)$ for all $w \in W_I$.
\end{lemma}
\begin{proof}
We compute $(s_{\alpha_i}(\mu))_I = (\mu - \langle \mu, \alpha_i^\vee \rangle \alpha_i)_I = \mu_I - \langle \mu, \alpha_i^\vee \rangle (\alpha_i)_I$. The first assertion now follows from  $(\alpha_i)_I = \alpha_i^I$ and $\langle \mu, \alpha_i^\vee \rangle = \langle \mu_I, {\alpha_i}^\vee \rangle$ for all $i \in I$. To obtain the second assertion, we write $w$ as a product of simple reflections and iteratively use the first part.
\end{proof}

Consider the restriction map $\res_I: \complex[P] \to \complex[P_I]$ defined by $\res_I(e^\mu) = e^{\mu_I}$. The following result states that this map commutes with the Demazure operators $T_i$ for $i \in I$.
\begin{proposition}\label{prop:commrel}
$\res_I \circ T_i = T_i \circ \res_I$ for all $i \in I$.
\end{proposition}
\begin{proof}
Let $\mu \in P$ and $i \in I$. Let $k = \langle \mu, \alpha_i^\vee \rangle = \langle \mu_I, \alpha_i^\vee \rangle$. If $k\geq 0$, then $T_i(e^\mu) = \sum_{p=0}^{k} e^{\mu-p\alpha_i}$ and $T_i (e^{\mu_I}) = \sum_{p=0}^k e^{\mu_I-p\alpha_i^I}$. Since $\res_I(e^{\alpha_i}) = e^{\alpha_i^I}$, this establishes the result. The $k<0$ case is similar.
   
\end{proof}

\begin{corollary}\label{cor:rest}
Let $\mu \in P^+$, $w \in W$ with $I:=\supp(w)$. Then
\begin{equation}\label{eq:demres}
    \dchar[\mu_I;\, \lie g_I]= \res_I (\dchar).
\end{equation}
Equivalently
\[d_{w,\mu}(\alpha) = d_{w,\mu_I}(\alpha_I; \,\lie g_I) \text{ for all }\alpha \in Q_I.\]
\end{corollary}
\begin{proof}
Proposition~\ref{prop:commrel} implies that $\res_I(T_w(e^\mu)) = T_w(e^{\mu_I}) = \dchar[\mu_I;\, \lie g_I]$, establishing \eqref{eq:demres}. The second assertion now follows easily from this.
 \end{proof}

\subsection{} We recall that the Kostant--Kumar module may be described alternately as follows \cite{KRV}:
\[ K(\lambda, w, \mu) = \U\lie g(v_{\lambda} \otimes V_w(\mu)).\]
The character of $K(\lambda, w, \mu)$ was computed by Kumar \cite{Kumar}.
\begin{thm} (Kumar) Let $w \in W$ and $\lambda, \mu \in P^+$. Then
  \begin{equation}\label{eq:charkk}
    \ch K(\lambda,w,\mu) = \dem_{\wnot}(e^\lambda \, \dchar) = \frac{\displaystyle\sum_{w \in W} \sgn(w) \, w\left( e^{\lambda + \rho} \,\dchar\right)}{e^\rho \displaystyle\prod_{\alpha \in \Delta^+} (1-e^{-\alpha})}
    \end{equation}
where $\sgn$ denotes the sign character of $W$, $\wnot \in W$ is the longest element and $\rho$ is the Weyl vector. 
\end{thm}
\begin{corollary} \label{cor:lrwformula}
For $\lambda, \mu, \nu \in P^+$ and $w \in W$, we have
\[ \lrw = \sum_{u \in W} \sgn(u)\, d_{w,\mu}((\lambda+\mu+\rho)-u(\nu+\rho)). \]
  \end{corollary}
When $w=w_0$, this formula reduces to the classical formula for tensor product multiplicities that is attributed variously to Steinberg, Brauer--Klimyk or Racah-Speiser \cite{kumar-icm2010,fulton-harris}. We omit the proof of Corollary~\ref{cor:lrwformula}, which closely follows the derivation in \cite[Theorem 3.3]{kumar-icm2010} of the formula for $\lr$  from the Weyl character formula\footnote{We remark that the operator $D$ of \cite[Theorem 3.3]{kumar-icm2010} is precisely our $T_{w_0}$.}.

The following strengthening of the above corollary is our first main result:
\begin{proposition}\label{prop:strongcor}
    Let $\lambda, \mu, \nu \in P^+$ and $w \in W$ with $\supp(w) = I$. Then $\lrw=0$ unless $(\lambda+\mu-\nu) \in Q_I$. Further,
\begin{equation}\label{eq:strongeq}
\lrw = \sum_{u \in W_I} \sgn(u)\, d_{w,\mu}((\lambda+\mu+\rho)-u(\nu+\rho))
\end{equation}
i.e., with the sum ranging over $W_I$ rather than over all of $W$.
\end{proposition}
\begin{proof}
 Consider $\alpha = (\lambda+\mu+\rho)-u(\nu+\rho)$.
 In view of Lemma~\ref{lem:demwt}, the term $d_{w,\mu}(\alpha)>0$ only if $\alpha \in Q^+$ with $\supp(\alpha) \subset I$. Now rewriting
 \[ \alpha = (\lambda+\mu-\nu) + \left(\nu+\rho - u(\nu+\rho)\right).\]
If $\lrw >0$, then $\lr >0$ and hence $\lambda+\mu-\nu \in Q^+$. Lemma~\ref{lem:simplem} implies that the second summand in the above equation is also in $Q^+$. Thus, $\supp(\alpha) \subset I$ implies that $\supp(\lambda+\mu-\nu) \subset I$, and by another appeal to Lemma~\ref{lem:simplem}, that $\supp(u) =\supp\left(\nu+\rho - u(\nu+\rho)\right) \subset I$. 
\end{proof}
Proposition~\ref{prop:strongcor} readily implies $w$-refined analogues of many results for $\lr$. For instance, the following result is a generalization of  \cite[Corollary (3.4)]{kumar-icm2010}.
\begin{corollary}
Let $\lambda, \mu, \nu \in P^+$ and $w \in W$ with $I=\supp(w)$.
    Let $\mathcal{P}=\mathcal{P}(\lambda,w,\mu)$ denote the set of weights of the $\lie b$-module $v_\lambda \otimes V_w(\mu)$ and let
    \begin{align*}
        n(\mathcal{P}) &= \min\{\langle \gamma, \alpha_i^\vee \rangle: \gamma \in \mathcal{P}, i \in I\} \in \integers, \\
        n(\nu) &= \min\{\langle \nu, \alpha_i^\vee \rangle: i \in I\} \geq 0.
        \end{align*}
    If $n(\mathcal{P}) + n(\nu) +1 \geq 0$, then $\lrw = \dim (V_w(\mu))_{\nu-\lambda}$.
\end{corollary}
 \begin{proof}
It suffices to prove that the summand $d_{w,\mu}((\lambda+\mu+\rho)-u(\nu+\rho))$ in \eqref{eq:strongeq} vanishes when $u \neq 1$. Let $\alpha=\lambda+\mu+\rho-u(\nu+\rho)$. By definition, $d_{w,\mu}(\alpha)$ vanishes precisely when $\mu-\alpha$ is not a weight of $V_w(\mu)$, or equivalently when $\lambda+(\mu-\alpha) \notin \mathcal{P}$.

Since $\nu+\rho$ is regular dominant and $u \in W_I, u \neq 1$, there exists $i \in I$ such that \[\langle u(\nu+\rho), \alpha_i^\vee\rangle = \langle \nu+\rho, u^{-1}(\alpha_i^\vee)\rangle <0.\] Thus $u^{-1}(\alpha_i^\vee)$ is a negative coroot, and for any $j \in I$ in its support we obtain
\[\langle \nu+\rho, u^{-1}(\alpha_i^\vee)\rangle \leq -\langle \nu+\rho, \alpha_j^\vee\rangle \leq -n(\nu)-1 \leq n(\mathcal{P}).\] 
Thus $\langle u(\nu+\rho)-\rho, \alpha_i^\vee\rangle \leq n(\mathcal{P})-1$. This implies by minimality of $n(\mathcal{P})$ that $u(\nu+\rho)-\rho \notin \mathcal{P}$.  Since $\lambda+(\mu-\alpha) = u(\nu+\rho)-\rho$, the proof is complete.
\end{proof}
\begin{remark}
    If $n(\mathcal{P}) \geq -1$, then $n(\mathcal{P})+n(\nu)+1 \geq 0$ for all $\nu \in P^+$ and the conclusion holds uniformly for all $\nu$. This reduces to the formulation of  \cite[Corollary (3.4)]{kumar-icm2010} for the $w=w_0$ case.
\end{remark}
\subsection{}
Given $\lambda, \mu, \nu \in P^+$ and $w \in W$ with $I=\supp(w)$, define $\delta_{\lambda\mu}^\nu(w) = 1$ if $(\lambda+\mu-\nu) \in Q_I$ and  $0$ otherwise. 
\begin{thm}\label{thm:lrwres}
Let $\lambda, \mu, \nu \in P^+$ and $w \in W$, with $I=\supp(w)$. Then
\begin{equation}\label{eq:lrlr-res}
    \lrw \;=\; \delta_{\lambda\mu}^\nu(w) \, c_{\lambda_I,\, \mu_I}^{\nu_I}(w; \, \lie g_I)
\end{equation}
\end{thm}
\begin{proof}
    If $\delta_{\lambda\mu}^\nu(w)=0$, then $\lrw=0$ by Proposition~\ref{prop:strongcor}. We may assume that $\delta_{\lambda\mu}^\nu(w)=1$, i.e., $\lambda+\mu-\nu \in Q_I$. This implies that $\alpha = (\lambda+\mu-\nu) + \left(\nu+\rho - u(\nu+\rho)\right) \in Q_I$. By Corollary~\ref{cor:rest}, we obtain $d_{w,\mu}(\alpha) = d_{w,\mu_I}(\alpha_I; \, \lie g_I)$. From Lemma~\ref{lem:sres}, we have $\left(u(\nu+\rho)\right)_I = u(\nu_I + \rho_I)$. Putting these together in Proposition~\ref{prop:strongcor}, $\lrw$ becomes
    \[\sum_{u \in W_I} \sgn(u)\, d_{w,\mu_I}((\lambda_I+\mu_I+\rho_I)-u(\nu_I+\rho_I))\]
    which coincides precisely with $c_{\lambda_I,\, \mu_I}^{\nu_I}(w; \, \lie g_I)$. 
    \end{proof}
    We also have the following lifting lemma that establishes a converse to the above statement.
    \begin{lemma}\label{lem:convlift}
        Let $(\lambda', \mu', \nu') \in P^+(\lie g_I)^3$ such that $(\lambda'+\mu'-\nu') \in Q^+(\lie g_I)$. Then, there exists $(\lambda,\mu,\nu) \in (P^+)^3$ with 
        $(\lambda+\mu-\nu) \in Q^+_I$ and  $\lambda_I=\lambda', \,\mu_I=\mu', \,\nu_I=\nu'$. Further, $\lrw = c_{\lambda'\mu'}^{\nu'}(w; \, \lie g_I)$. 
    \end{lemma}
    \begin{proof}
Let $\alpha=\lambda'+\mu'-\nu' =\sum_{i \in I} c_i \alpha_i \,\in Q(\lie g_I)$. We may view $\alpha$ as an element of $Q_I$ under our identification $Q_I \cong Q(\lie g_I)$. Define $\lambda \in P^+$ via $\langle \lambda, \alpha_j^\vee \rangle =\langle \lambda', \alpha_j^\vee \rangle$ if $j\in I$ and zero if $j \in S\backslash I$. Likewise define $\mu\in P^+$ via $\langle \mu, \alpha_j^\vee \rangle =\langle \mu', \alpha_j^\vee \rangle$ if $j\in I$ and zero otherwise. Finally, let \[\nu = \lambda+\mu - \alpha.\]
We note that $\nu_I = \lambda_I + \mu_I - \alpha_I =\nu'$ by definition of $\alpha$. Since $\supp(\alpha) \subset I$, we have $\langle \alpha, \alpha_j^\vee \rangle \leq 0$ for all $j \notin I$. This proves $\nu \in P^+$ and completes this argument. The equality $\lrw = c_{\lambda'\mu'}^{\nu'}(w; \, \lie g_I)$ now follows from Theorem~\ref{thm:lrwres}.
    \end{proof}
Thus, the nonzero $w$-refined tensor product multiplicities of the ambient algebra $\lie g$ coincide with those of the subalgebra $\lie g_I$ for $I=\supp(w)$.  
\subsection{}\label{sec:wcomps}
Viewing $I:=\supp(w)$ as a subset of the nodes of the Dynkin diagram of $\lie g$, we may identify it with the induced subgraph it defines.  
Suppose $I_j \subseteq S$ ($j=1, \cdots, r$) are the  connected components of $I$.
Let $W_j$ denote the subgroup of $W$ generated by $\{s_i: i \in I_j\}$. We can write $w = \prod_{j=1}^r w_j$ where $w_j \in W_j$; note that $w_j$ and $w_k$ commute for $j \neq k$. 

Let $\lie g_j = \lie g_{I_j}$ and $\lie h_j = \lie h_{I_j}$. For the weights $\lmn, \rho$ of $\lie g$, let $\lambda_j, \mu_j, \nu_j, \rho_j$ denote their restrictions to $\lie h_j$.  We now have the following important proposition.
\begin{proposition}\label{prop:w1w2mult}
    Let notation be as above. Then 
    \[ \lrw = \delta_{\lambda\mu}^\nu(w) \, \prod_{i=1}^r c_{\lambda_i\mu_i}^{\nu_i}(w_i; \lie g_i) \]
    where $\delta_{\lambda\mu}^\nu(w) = 1$ if $(\lambda + \mu - \nu) \in Q_I$ and $0$ otherwise.
\end{proposition}
\begin{proof}
Noting that $\lie g_I$ is in general semisimple, it is easily seen that Proposition~\ref{prop:w1w2mult} follows from Theorem~\ref{thm:lrwres} and Equation~\eqref{eq:cprod}.
 \end{proof}
\subsection{}\label{sec:combexcellent}
The $\lrw$ are also related to the {\em combinatorial excellent filtrations} of Demazure modules and have descriptions in terms of Lakshmibai-Seshadri paths \cite{LLM} or crystals \cite{Joseph1, Joseph-birthday}.
We briefly recall the Joseph decomposition rule for Kostant--Kumar modules in terms of crystals \cite{Joseph-birthday,KRV}. We will apply these notions in type $A$ in the later sections of the paper.

We refer the reader to \cite{bump-schilling} for the general theory of crystals of finite dimensional representations of $\lie g$. For each $\lambda \in P^+$, let $\crys(\lambda)$ denote the crystal corresponding to the irreducible representation $V(\lambda)$; this is the (unique upto isomorphism) connected crystal with an element $b_\lambda$ satisfying
\[\wt b_\lambda = \lambda \text{ and } e_i b_\lambda = 0  \text{ for all } i \in S.\]
Given $w \in W$, let $\crys_w(\lambda)$ denote the Demazure crystal indexed by $w$. This is defined by:
\begin{equation}\label{eq:demcrygentype}
  \crys_w(\lambda) :=\{f_{i_1}^{m_1} f_{i_2}^{m_2} \cdots f_{i_r}^{m_r} b_\lambda: m_j \geq 0\}
  \end{equation}
where $w=s_{i_1} s_{i_2} \cdots s_{i_r}$ is a reduced expression for $w$. Here $e_i, f_i$ denote the crystal raising and lowering operators. We note that for $b \in \crys_w(\lambda)$, its weight $\wt b = \lambda - \sum_{j=1}^r m_{j} \alpha_{i_j}$. We have
\begin{equation}\label{eq:lamwtb}
(\lambda - \wt b) \in \integers_+\text{-span} \{\alpha_i: i \in \supp(w)\}.
\end{equation}
The character of $\crys_w(\lambda)$ coincides with that of the Demazure module $V_w(\lambda)$. 

Given $\lambda, \mu \in P^+$ and $w \in W$, the decomposition rule for the Kostant--Kumar module $K(\lambda,w,\mu)$ was given by Joseph (for simply-laced $\lie g$; see \cite{KRV} for the non-simply laced case) in \cite[Theorem 5.25, (4.6*)]{Joseph-birthday}:
\begin{equation}\label{eq:lskkdecomp}
  K(\lambda,w,\mu) = \bigoplus V(\wt b)
\end{equation}
where the sum runs over $b \in b_\lambda \otimes \crys_w(\mu)$ such that $e_i b  =0$ for all  $i \in S$. 
We thereby obtain: 
\begin{proposition}
Let $\lie g$ be a finite-dimensional simple Lie algebra with $S$ indexing its simple roots. Let $\lmn \in P^+, w \in W$. Then: 
\begin{equation}\label{eq:lrwcrys}
\lrw = \# \{b \in b_\lambda \otimes \crys_w(\mu):  e_i b  =0 \; \forall  i \in S \text{ and } \wt b = \nu\} 
\end{equation}
\end{proposition}

\begin{corollary}\label{cor:scaling}
    Let $\lambda, \mu, \nu \in P^+$ and $w \in W$. If $\lrw >0$, then $\lrwk>0$ for all $k \geq 1$.
\end{corollary}
\begin{proof}
   If $\lrw>0$, it follows from \eqref{eq:lrwcrys} that there exists $b' \in \crys_w(\mu)$ such that $b=b_\lambda \otimes b'$ satisfies 
    $e_i b  =0$ for all  $i \in S$ and $\wt b = \nu$. Suppose $b'=f_{i_1}^{m_1} f_{i_2}^{m_2} \cdots f_{i_r}^{m_r} b_\mu$. For $k \geq 1$, define $b'_k:=f_{i_1}^{km_1} f_{i_2}^{km_2} \cdots f_{i_r}^{km_r} \,b_{k\mu} \in \crys_w(k\mu)$. It is easily checked that \[b_k:=b_{k\lambda} \otimes b'_k\] satisfies $e_i b  =0$ for all  $i \in S$ and $\wt b = k\nu$. Thus $\lrwk>0$.
\end{proof}
\begin{remark}\label{rem:scaling}
    In the above proof, one can also show (for instance via Littelmann's path crystal \cite{LittelmannAnnals}) that the map $b' \mapsto b'_k$ is injective for each $k \geq 1$, and thereby establish that $\lrwk \geq \lrw$. We give an alternate proof of this fact below in type $A$ using the hive model (Corollary~\ref{cor:corscaling}).
\end{remark}

\section{Saturation and semigroup properties}\label{sec:semigpsatdef}
Let $\lie g$ be a finite-dimensional semisimple Lie algebra.
\subsection{}
 Let $\mct$ denote the set of triples $(\lambda, \mu, \nu) \in (P^+)^3$ such that $\lambda+\mu - \nu \in Q$.
 \begin{definition}\label{def:semigp}
   An element $w \in W$ is said to have the {\em semigroup property} if 
\begin{equation} \label{eq:semigpprop-def}
c_{\lambda_1,\,\mu_1}^{\,\nu_1}(w) >0 \text{ and } c_{\lambda_2,\,\mu_2}^{\,\nu_2}(w) >0 \text{ implies } c_{{\lambda_1+\lambda_2},\, {\mu_1+\mu_2}}^{\,\nu_1+\nu_2}(w) >0
\end{equation}
for all $(\lambda_i, \mu_i, \nu_i) \in \mct$, $i=1,2$.
\end{definition}
We will also have occasion to consider the following stronger variant:
\begin{definition}\label{def:semigp-strong}
   An element $w \in W$ is said to have the {\em strong semigroup property} if \\
$c_{\lambda_1,\,\mu_1}^{\,\nu_1}(w) >0$  and  $c_{\lambda_2,\,\mu_2}^{\,\nu_2}(w) >0$ implies  \[c_{{\lambda_1+\lambda_2},\, {\mu_1+\mu_2}}^{\,\nu_1+\nu_2}(w) \geq \max\left(c_{\lambda_1,\,\mu_1}^{\,\nu_1}(w),\; c_{\lambda_2,\,\mu_2}^{\,\nu_2}(w)\right)\]
for all $(\lambda_i, \mu_i, \nu_i) \in \mct$, $i=1,2$.
\end{definition}

\begin{definition}\label{def:sat}
   An element $w \in W$ is said to have the {\em saturation property} if the following holds for all $(\lambda, \mu, \nu) \in \mct$:
\begin{equation} \label{eq:satprop}
\lrwk > 0 \text{ for some } k\geq 1 \text{ implies } \lrw > 0.
\end{equation}
More generally, an integer $d \geq 1$ is said to be a {\em saturation factor} for $w \in W$ if the following holds for all $(\lambda, \mu, \nu) \in \mct$:
\begin{equation} \label{eq:satprop-factor}
\lrwk > 0 \text{ for some } k\geq 1 \text{ implies } \lrwd > 0.
\end{equation}
\end{definition}
\begin{remark}\label{rem:satsat}
In view of Corollary~\ref{cor:scaling}, if $d$ is a saturation factor for $w$, then so is any positive integer multiple of $d$. 
\end{remark}
Kapovich-Millson \cite{kapovich-millson} established that $w=w_0$ (i.e., the tensor product multiplicity case) has a saturation factor for all finite-dimensional semisimple Lie algebras. Their saturation factor is $k_{\lie g}^2$ where $k_{\lie g}$ is the LCM of the coordinates of the highest root of $\lie g$ in the basis of simple roots. We conjecture the existence result for all $w$:

\begin{conjecture}\label{conj:sat-fact-exists}
Let $\lie g$ be a finite-dimensional semisimple Lie algebra and let $w \in W(\lie g)$. Then there exists a saturation factor for $w$.
\end{conjecture}

We note that $w=1$ has the saturation and strong semigroup properties since $\lrw[\id] = \delta_{\lambda+\mu, \,\nu}$. It is well-known that $w_0$ has the semigroup property in all types \cite{Zel,elashvili}. In fact, it also has the strong semigroup property, as follows from \cite[Lemma (3.9)]{kumar-icm2010}.
Finally, in type $A$, $w=\wnot$  has the saturation property - this is exactly the  Knutson--Tao saturation theorem \cite{KTHoneycombs}. 

In type $A$, we establish below (Theorem~\ref{thm:mainthm-312}) that a certain class of pattern-avoiding permutations  $w$ have both the saturation and strong semigroup properties.
\subsection{}
Given $w \in W$, let $I=\supp(w)$. We may consider the semisimple Lie algebra $\lie g_I$ and view $w$ as an element of its Weyl group $W_I=W(\lie g_I)$. 
 
\begin{proposition}\label{prop:ambient-vs-res}
Let $w \in W$ with $I=\supp(w)$. Then  $w \in W$ has the saturation  property for the ambient Lie algebra $\lie g$ if and only if $w \in W(\lie g_I)$ has the saturation  property for $\lie g_I$. Likewise, $w \in W$ has the semigroup (resp. strong semigroup) property if and only if $w \in W(\lie g_I)$ has the semigroup (resp. strong semigroup) property.
\end{proposition}
\begin{proof}
This follows from Theorem~\ref{thm:lrwres} and the linearity of the lifts constructed in Lemma~\ref{lem:convlift}. 
\end{proof}
\begin{example}
    Let $\lie g = B_n$ for $n \geq 2$. We recall in this case that $w_0$ does not have the saturation property \cite{kumar-icm2010, elashvili}. Nevertheless, the proposition above gives us many examples of $w \in W(B_n)$ which have the saturation property - let $I\subset S$ be any type $A$ subdiagram of $B_n$ and take $w$ to be the longest element of $W(\lie g_I) = W_I \subset W$. Proposition~\ref{prop:ambient-vs-res} and Knutson-Tao's saturation theorem together imply that $w$ has the saturation property.  In view of Theorem~\ref{thm:mainthm-312} below, we can take $w$ more generally to be a $312$- or $231$-avoiding permutation in $W_I$.
    Similar examples may be constructed for $\lie g$ of other types as well.
\end{example}

We will need below the following elementary lemma which compares the product of maxima with the maximum of products. We leave the easy proof to the reader.
\begin{lemma}\label{lem:multmax}
Let $a_i, b_i$ be positive real numbers for $i=1, 2, \ldots, r$. Then \[ \displaystyle\prod_{i=1}^r \max(a_i, b_i) \geq \max(\prod_{i=1}^r a_i, \; \prod_{i=1}^r b_i)\] 
\end{lemma}
We now have the following proposition: 
 \begin{proposition}
     \label{prop:w1w2cor}
Let $w \in W$ with $I=\supp(w)$. Let $I_j, \, j=1, \ldots, r$ denote the connected components of $I$ and let $w=\prod_{j=1}^r w_j$ with $\supp(w_j) = I_j$. 
\begin{enumerate}
    \item If each $w_j \in W(\lie g_{I_j})$ has the semigroup (resp. strong semigroup) property, then $w\in W$ has the semigroup (resp. strong semigroup) property. 
    \item If $d_j$ is a saturation factor for $w_j$, then their least common multiple $d:=lcm(\{d_j: j=1, \ldots, r\})$ is a saturation factor for $w$.
    In particular, $w$ has the saturation property if each $w_j$ does. 
    \end{enumerate}
\end{proposition}
\begin{proof}
The first assertion is a direct consequence of Proposition ~\ref{prop:w1w2mult} and Lemma~\ref{lem:multmax}. For (2), since $d$ is a multiple of $d_j$, it follows from Remark~\ref{rem:satsat} that $d$ is a saturation factor for $w_j$ for all $j=1, \ldots, r$. Appealing again to Proposition~\ref{prop:w1w2mult} completes the argument.
\end{proof}

\section{Demazure crystals in type $A$}\label{sec:demcrys}
The rest of the paper is concerned with the type $A$ case, i.e., $\lie g = \mathfrak{sl}_n$.
\subsection{The Tableaux model}

We now specialize the results of Section~\ref{sec:combexcellent} to the type $A$ case. We henceforth assume that $\lie g = \mathfrak{sl}_n\complex$. In this case, the set of simple roots is indexed by $S = \{1, 2, \ldots, n-1\}$. We also have
$W \cong S_n$ (the symmetric group) with simple transpositions as generators $s_i = (i \; i+1)$. We also have that $P \cong \integers_+^n$ and $P^+$ may be identified with the set $\pn$ of partitions with at most $n$ parts\footnote{Here we identify, as usual, tuples in $\integers_+^n$ which differ by a multiple of $(1,1,\ldots,1)$.}. We also let $\varepsilon_i - \varepsilon_j, i \neq j$ denote the set of all roots of $\lie g$. The set of positive roots is:
\[ \Phi^+ = \{ \varepsilon_i - \varepsilon_j, 1 \leq i < j \leq n\}\]

Given a partition $\mu \in \pn$, let $\Tab(\mu)$ denote the set of semistandard Young tableaux of shape $\mu$ with entries in $1, 2, \cdots, n$. To each $T \in \Tab(\mu)$ we associate its reverse row reading word $b_T$ obtained by reading the entries of $T$ from right to left and top to bottom (in English notation), for example, 

\begin{equation}\label{eq:tabex}
    T = \ytableaushort{123,23} \ \text{and} \ b_T = 32132 
\end{equation}
We will often identify a tableau with its reverse row reading word. The crystal raising and lowering operators $e_i, f_i \, (1 \leq i < n)$ act on the set $\Tab(\mu)$, or more precisely, on the associated words. More generally, $e_i, f_i$ have an action on the set of all words in the alphabet $\{1, 2, \cdots, n\}$ \cite[Chapter 5]{lothaire}. 

There exists a unique tableau $\supstd[\mu]$ in $\Tab(\mu)$ such that $e_i \supstd[\mu]=0 $ for all $i \in S$ and a unique tableau $\supstdual[\mu]$ such that $f_i \supstdual[\mu] = 0$ for all $i \in S$. The weight of a word $u$ in the alphabet $\{1, 2, \cdots, n\}$ is the tuple $(a_1, a_2, \cdots, a_n)$ where $a_i$ is the number of occurrences of $i$ in $u$. For a tableau $T \in \Tab(\mu)$ we define the weight of $T$ to be the weight of the corresponding word $b_T$.

Then $\supstd[\mu]$ is the highest weight element and $\supstdual[\mu]$ is the lowest weight element in $\Tab(\mu)$.

\begin{definition}\label{def:domword}
   A word $u = u_1u_2\cdots u_k$ in the alphabet $\{1,2,\cdots,n\}$  is said to be {\em dominant} (or a ballot sequence) if every left subword $u^{(t)} = u_1u_2\cdots u_t$ of $u$ for $1 \leq t \leq k$, contains more occurrences of $i$ than $i+1$ for all $1 \leq i <n$, or equivalently, if $e_i u =0$ for all $1 \leq i <n$.
\end{definition}

Given $w \in S_n$, fix a reduced decomposition $w = s_{i_1} s_{i_2} \cdots s_{i_k}$. For $\mu \in \pn$,
the Demazure crystal $\crys_w(\mu)$ (in the tableaux model) is given by:
\begin{equation}\label{eq:demcry}
  \crys_w(\mu) :=\{f_{i_1}^{m_1} f_{i_2}^{m_2} \cdots f_{i_k}^{m_k} \supstd[\mu]: m_j \geq 0\}
  \end{equation}
For $w=w_0$, we obtain $\crys_{w_0}(\mu) = \Tab(\mu)$. 
We likewise define the {\em opposite Demazure crystal} $\crys_w^{op}(\mu)$. Fix a reduced decomposition $ww_0 = s_{i_1} s_{i_2} \cdots s_{i_k}$ and define:
\begin{equation}\label{eq:opdemcry}
  \crys_w^{op}(\mu) :=\{e_{i_1}^{m_1} e_{i_2}^{m_2} \cdots e_{i_k}^{m_k} \supstdual[\mu]: m_j \geq 0\}
\end{equation}

Equation~\eqref{eq:lskkdecomp} now becomes the following statement in type $A$:
\begin{thm}\label{thm:joseph}  
$\lrw$ is the cardinality of the set \[ \demcrys_\lambda^\nu(\mu, w) :=\{T \in \crys_w(\mu): \supstd[\lambda]*b_T \text{ is a dominant word of weight } \nu\}.\] Here $*$ denotes concatenation of words.
\end{thm}

\noindent
We also let \[\Tab_\lambda^\nu(\mu):= \demcrys_\lambda^\nu(\mu,w_0) = \{T \in \Tab(\mu): \supstd[\lambda]*b_T \text{ is a dominant word of weight } \nu\}.\]

\begin{figure}
\begin{center}
\begin{tikzpicture}[x={(1cm*0.5,-\rootthree cm*0.5)},y={(1cm*0.5,\rootthree cm*0.5)}]
\foreach\i in{0,...,4}
\foreach\j in{\i,...,4}{
  \pgfmathtruncatemacro{\k}{5 - \j + \i};
  \pgfmathtruncatemacro{\l}{\i + 1};
  \draw(\i,\j)node(a\i\j){$a_{\k\l}$};
}
  \foreach\i/\ii in{0/-1,1/0,2/1,3/2,4/3}
 \foreach\j/\jj in{0/-1,1/0,2/1,3/2,4/3}{
  \ifnum\i<\j \draw[color=red](a\i\jj)--(a\i\j); \fi
  \ifnum\i>\j\else\ifnum\i>0 \draw[color=blue](a\ii\j)--(a\i\j);\fi\fi
 }
\end{tikzpicture}\qquad
\begin{tikzpicture}[x={(1cm*0.5,-\rootthree cm*0.5)},y={(1cm*0.5,\rootthree cm*0.5)}]
\foreach\i in{0,...,4}
\foreach\j in{\i,...,4}{
  \pgfmathtruncatemacro{\k}{5 - \j + \i};
  \pgfmathtruncatemacro{\l}{\i + 1};
  \draw(\i,\j)node(a\i\j){$\bullet$};
}
\foreach\i/\ii in{0/-1,1/0,2/1,3/2,4/3}
 \foreach\j/\jj in{0/-1,1/0,2/1,3/2,4/3}{
  \ifnum\i<\j \draw[color=red](a\i\jj)--(a\i\j); 
  \fi;
 }
  \node at (-.3,.3) {$s_4$}; \node at (-.3,1.3) {$s_3$}; \node at (-.3,2.3) {$s_2$}; \node at (-.3,3.3) {$s_1$};
 
 \node at (.7,1.3) {$s_3$}; \node at (.7,2.3) {$s_2$}; \node at (.7,3.3) {$s_1$};
 
 \node at (1.7,2.3) {$s_2$}; \node at (1.7,3.3) {$s_1$};
 
 \node at (2.7,3.3) {$s_1$};
\end{tikzpicture}
\caption{Gelfand--Tsetlin array for $n=5$. The red edges $a_{ij} \rightarrow a_{i-1,j}$ are labelled by $s_{i-j}$}
\label{fig:gt-array-aij-specified}
\end{center}
\end{figure}

\subsection{Gelfand--Tsetlin polytopes}\label{sec:gt}
A word on notation for the remainder of the paper: if $\mathcal{P}$ is a polyhedron, or a subset thereof, then $\mathcal{P}_\integers$ will denote the set of integer points in $\mathcal{P}$. 

A Gelfand--Tsetlin (GT) pattern of size $n$ is a triangular array $A=(a_{ij})_{n \geq i \geq j \geq 1} \in \reals^{n(n+1)/2}$ of real numbers (figure~\ref{fig:gt-array-aij-specified}) satisfying the following (``North-East'' and ``South-East'') inequalities for all $i > j$: 
\begin{align*}
    \NE_{ij}(A) &:= a_{ij} - a_{i-1, j} \geq 0,  \\ \SE_{ij}(A) &:= a_{i-1, j} - a_{i, j+1} \geq 0.
\end{align*}
    For $\mu \in \pn$, the GT polytope $\GT(\mu)$ is the set of all GT patterns with $a_{ni} = \mu_i$ for $1 \leq i \leq n$. For example, the following is a GT pattern of shape $\mu = (6,4,2,2,1)$.

\begin{center}
\begin{tikzpicture}[x={(1cm*0.5,-\rootthree cm*0.5)},y={(1cm*0.5,\rootthree cm*0.5)}]
  \draw(0,0)node(a00){$6$};\draw(0,1)node(a01){$5$};\draw(0,2)node(a02){$5$};\draw(0,3)node(a03){$3$};\draw(0,4)node(a04){$2$};
  \draw(1,1)node(a11){$4$};\draw(1,2)node(a12){$4$};\draw(1,3)node(a13){$3$};\draw(1,4)node(a14){$2$};
  \draw(2,2)node(a22){$2$};\draw(2,3)node(a23){$2$};\draw(2,4)node(a24){$1$};
  \draw(3,3)node(a33){$2$};\draw(3,4)node(a34){$1$};
  \draw(4,4)node(a44){$1$};

\end{tikzpicture}
\end{center}

We have the standard bijection $A \mapsto \bg(A)$ from $\GTz(\mu)$  to $\Tab(\mu)$, with the tableau $\bg(A)$ uniquely specified by the condition that for all $i \geq j$, the number of $\boxed{i}$ in row $j$ equals $\NE_{ij}(A)$ (with $a_{i-1, i} :=0$).  For example, letting $A$ be the above GT pattern of shape $\mu =(6,4,2,2,1)$, we have
\begin{equation}\label{eq:GTtab}
    \bg(A) = \ytableaushort{112335,2234,34,45,6}
\end{equation}

\subsection{Kogan faces}\label{sec:kogan}
Fix a subset $F \subseteq \{(i,j): n \geq i > j \geq 1\}$. The {\em Kogan face}\footnote{The corresponding version with SE inequalities would define a {\em dual Kogan face}.} of $\GT(\mu)$ associated to $F$ is :
\[ \kogan(\mu,F) = \{ A \in \GT(\mu): \NE_{ij}(A)=0 \text{ for } (i,j) \in F\} \]
Next, to each pair $n \geq i > j \geq 1$, associate the simple transposition $s_{i-j} \in S_n$ (figure~\ref{fig:gt-array-aij-specified}). We list the elements of $F$ in lexicographically increasing order: $(i,j)$ precedes $(i',j') \iff$ either $i<i'$, or $i=i'$ and $j<j'$. List the corresponding $s_{i-j}$ (written left to right) in this order; let $\sigma(F)$ denote their product in $S_n$. If $\len \sigma(F) = |F|$, we say that $F$ is {\em reduced} and set \cite[Definition 5.1]{fujita}:
\[ \varpi(F) = \wnot \,\sigma(F) \,\wnot \]
\begin{example}
Let $n=5$ and let $F_1 = \{(3,1),(5,2)\}$, $F_2 = \{(3,1),(4,2)\}$. Then $\sigma(F_1) = s_2s_3$ and $\sigma(F_2) = s_2^2 = 1$. Thus $F_1$ is reduced and $F_2$ is not.  
\end{example}
We now associate to each element of $S_n$ a union of Kogan faces as follows:
\begin{definition}\label{def:gtkogan}
  For $w \in S_n$ and $\mu \in \pn$, let
  \begin{equation}\label{eq:koganwdef}
    \kogan(\mu, w) := \bigcup \kogan(\mu, F), \text{ the union over reduced } F \text{ for which } \varpi(F) = w.
    \end{equation}
\end{definition}
\begin{example}
Let $n=5$ and $w=w_0(s_1s_2s_3)w_0 = s_4s_3s_2$. Then $F_1 = \{ (2,1), (3,1), (4,1)\}$, $F_2 = \{ (2,1), (3,1), (5,2)\}$, $F_3 = \{ (2,1), (4,2), (5,2)\}$, $F_4 = \{ (3,2), (4,2), (5,2)\}$ are all the reduced faces $F$ for which $\varpi(F) = w$. Thus $\kogan(\mu, w) = \bigcup_{i=1}^4 \kogan(\mu, F_i)$.
\end{example}
The following important proposition relates the Demazure crystals in the Gelfand--Tsetlin and tableaux models \cite[Corollary 5.19]{fujita}:
\begin{proposition}\label{prop:fujita}
(Fujita) The bijection $\bg: \GTz(\mu) \to \Tab(\mu)$ restricts to a bijection $\koganz(\mu, \wnot w) \stackrel{\sim}{\to} \crys_w(\mu)$. 
\end{proposition}

It was previously shown in \cite{smirnov-et-al} (for regular $\mu$) and \cite{PS} (for $w$ $312$-avoiding) that $\koganz(\mu, \wnot w)$ and $\crys_w(\mu)$  have the same character. This weaker statement is however inadequate for our present purposes.

\begin{remark}\label{rem:sigma}
  We can also put a different total order on the set $\{(i,j): n \geq i > j \geq 1\}$. Let $(i,j)$ precede $(i',j')$ if and only if $j<j'$, or $j=j'$ and $i<i'$. List the $(i,j) \in F$ in increasing order relative to this new total order and denote the product of the corresponding $s_{i-j}$  by $\sigma'(F)$. It can be easily checked that $\sigma(F) = \sigma'(F)$ in $S_n$. 
\end{remark}

\section{Hive-Kogan faces: a polytopal model for $\lrw$}\label{sec:hive-clmn}
\subsection{Hives}
We begin with a quick overview \cite{buch, Knutson-Tao-v2}. The {\em big hive triangle} $\bighive$ is the array of Figure~\ref{fig:bighive-rhombi}, with $(n+1)$ vertices on each boundary edge, and $(n-2)(n-1)/2$ interior vertices. We note that there are 3 types of rhombi in $\bighive$ (figure \ref{fig:bighive-rhombi}): the Northeast (NE) slanted (in red), the Southeast (SE) slanted (in green) and the vertical diamonds (in blue).

\begin{figure}
\begin{center}
\raisebox{9mm}{\begin{tikzpicture}[scale=0.9]
      \node (a00) at (0,0*1.732) {};
      \node (a01) at (0.5,0.5*1.732) {};
      \node (a02) at (1,1*1.732) {};
      \node (a03) at (1.5,1.5*1.732) {};
      \node (a04) at (2,2*1.732) {};
      \node (a05) at (2.5,2.5*1.732) {};
      \node (a11) at (1,0*1.732) {};
      \node (a12) at (1.5,0.5*1.732) {};
      \node (a13) at (2,1*1.732) {};
      \node (a14) at (2.5,1.5*1.732) {};
      \node (a15) at (3,2*1.732) {};
      \node (a22) at (2,0*1.732) {};
      \node (a23) at (2.5,0.5*1.732) {};
      \node (a24) at (3,1*1.732) {};
      \node (a25) at (3.5,1.5*1.732) {};
      \node (a33) at (3,0*1.732) {};
      \node (a34) at (3.5,0.5*1.732) {};
      \node (a35) at (4,1*1.732) {};
      \node (a44) at (4,0*1.732) {};
      \node (a45) at (4.5,0.5*1.732) {};
      \node (a55) at (5,0*1.732) {};

      \draw[-] [draw=red, ultra thick] (a00.center) -- (a11.center);
      \draw[-] [draw=black,thick] (a11.center) -- (a22.center);
      \draw[-] [draw=black,thick] (a22.center) -- (a33.center);
      \draw[-] [draw=black,thick] (a33.center) -- (a44.center);
      \draw[-] [draw=green,ultra thick] (a44.center) -- (a55.center);
      
      \draw[-] [draw=red,ultra thick] (a01.center) -- (a12.center);
      \draw[-] [draw=black,thick] (a12.center) -- (a23.center);
      \draw[-] [draw=black,thick] (a23.center) -- (a34.center);
      \draw[-] [draw=green, ultra thick] (a34.center) -- (a45.center);

      \draw[-] [draw=black,thick] (a02.center) -- (a13.center);
      \draw[-] [draw=blue,ultra thick] (a13.center) -- (a24.center);
      \draw[-] [draw=black,thick] (a24.center) -- (a35.center);
      
      \draw[-] [draw=black,thick] (a03.center) -- (a14.center);
      \draw[-] [draw=black,thick] (a14.center) -- (a25.center);
      
      \draw[-] [draw=black,thick] (a04.center) -- (a15.center);
      \draw[-] [draw=red, ultra thick] (a00.center) -- (a01.center);
      \draw[-] [draw=black,thick] (a01.center) -- (a02.center);
      \draw[-] [draw=black,thick] (a02.center) -- (a03.center);
      \draw[-] [draw=black,thick] (a03.center) -- (a04.center);
      \draw[-] [draw=black,thick] (a04.center) -- (a05.center);
      
      \draw[-] [draw=red,ultra thick] (a11.center) -- (a12.center);
      \draw[-] [draw=black,thick] (a12.center) -- (a13.center);
      \draw[-] [draw=blue,ultra thick] (a13.center) -- (a14.center);
      \draw[-] [draw=black,thick] (a14.center) -- (a15.center);

      \draw[-] [draw=black,thick] (a22.center) -- (a23.center);
      \draw[-] [draw=blue,ultra thick] (a23.center) -- (a24.center);
      \draw[-] [draw=black,thick] (a24.center) -- (a25.center);
      
      \draw[-] [draw=black,thick] (a33.center) -- (a34.center);
      \draw[-] [draw=black,thick] (a34.center) -- (a35.center);
      
      \draw[-] [draw=green,ultra thick] (a44.center) -- (a45.center);
      \draw[-] [draw=red,ultra thick] (a01.center) -- (a11.center);
      
      \draw[-] [draw=black,thick] (a02.center) -- (a12.center);
      \draw[-] [draw=black,thick] (a12.center) -- (a22.center);
      
      \draw[-] [draw=black,thick] (a03.center) -- (a13.center);
      \draw[-] [draw=blue, ultra thick] (a13.center) -- (a23.center);
      \draw[-] [draw=black,thick] (a23.center) -- (a33.center);

      \draw[-] [draw=black,thick] (a04.center) -- (a14.center);
      \draw[-] [draw=blue,ultra thick] (a14.center) -- (a24.center);
      \draw[-] [draw=black,thick] (a24.center) -- (a34.center);
      \draw[-] [draw=green, ultra thick] (a34.center) -- (a44.center);
      
      \draw[-] [draw=black,thick] (a05.center) -- (a15.center);
      \draw[-] [draw=black,thick] (a15.center) -- (a25.center);
      \draw[-] [draw=black,thick] (a25.center) -- (a35.center);
      \draw[-] [draw=black,thick] (a35.center) -- (a45.center);
      \draw[-] [draw=green,ultra thick] (a45.center) -- (a55.center);
        \fill[red,opacity=0.2] (a00.center)--(a01.center)--(a12.center)--(a11.center)--cycle;
        \fill[green,opacity=0.2] (a34.center)--(a44.center)--(a55.center)--(a45.center)--cycle;
        \fill[blue,opacity=0.2] (a14.center)--(a24.center)--(a23.center)--(a13.center)--cycle;
    \end{tikzpicture}}\hspace{1cm}
    \begin{tikzpicture}[scale=1]
      \node (a00) at (0,0*1.732) {$\bullet$};
      \node (a01) at (0.5,0.5*1.732) {$\bullet$};
      \node (a02) at (1,1*1.732) {$\bullet$};
      \node (a03) at (1.5,1.5*1.732) {$\bullet$};
      \node (a04) at (2,2*1.732) {$\bullet$};
      \node (a05) at (2.5,2.5*1.732) {$\bullet$};
      \node (a11) at (1,0*1.732) {$\bullet$};
      \node (a12) at (1.5,0.5*1.732) {};
      \node (a13) at (2,1*1.732) {};
      \node (a14) at (2.5,1.5*1.732) {};
      \node (a15) at (3,2*1.732) {$\bullet$};
      \node (a22) at (2,0*1.732) {$\bullet$};
      \node (a23) at (2.5,0.5*1.732) {};
      \node (a24) at (3,1*1.732) {};
      \node (a25) at (3.5,1.5*1.732) {$\bullet$};
      \node (a33) at (3,0*1.732) {$\bullet$};
      \node (a34) at (3.5,0.5*1.732) {};
      \node (a35) at (4,1*1.732) {$\bullet$};
      \node (a44) at (4,0*1.732) {$\bullet$};
      \node (a45) at (4.5,0.5*1.732) {$\bullet$};
      \node (a55) at (5,0*1.732) {$\bullet$};
       \node (la5) at (-0.5,0*1.732) {$|\lambda|$};
       \node (la4) at (-0.6,0.5*1.732) {$\sum_{i=1}^4 \lambda_i$};
       \node (la3) at (-0.1,1*1.732) {$\sum_{i=1}^3 \lambda_i$};  
       \node (la2) at (0.4,1.5*1.732) {$\sum_{i=1}^2 \lambda_i$};    
       \node (la1) at (1.5,2*1.732) {$\lambda_1$};
       \node (zz) at (2.5,2.8*1.732) {$0$};      

       \node (nu1) at (3.5,2*1.732) {$\nu_1$};
       \node (nu2) at (4.4,1.5*1.732) {$\sum_{i=1}^2 \nu_i$};
       \node (nu3) at (5,1*1.732) {$\sum_{i=1}^3 \nu_i$};
       \node (nu4) at (5.6,0.5*1.732) {$\sum_{i=1}^4 \nu_i$};
       \node (nu5) at (6.7,0*1.732) {$|\nu| = |\lambda| + |\mu|$};
       \node (mu1) at (0.7,-0.4*1.732) {$\substack{\mu_1\\+|\lambda|}$};
       \node (mu2) at (1.8,-0.4*1.732) {$\substack{\sum_{i=1}^2 \mu_i\\+|\lambda|}$};
       \node (mu3) at (3.2,-0.4*1.732) {$\substack{\sum_{i=1}^3 \mu_i\\+|\lambda|}$};
       \node (mu4) at (4.5,-0.4*1.732) {$\substack{\sum_{i=1}^4 \mu_i\\+|\lambda|}$};

      \draw[-] [draw=black, thick] (a00.center) -- (a11.center);
      \draw[-] [draw=black,thick] (a11.center) -- (a22.center);
      \draw[-] [draw=black,thick] (a22.center) -- (a33.center);
      \draw[-] [draw=black,thick] (a33.center) -- (a44.center);
      \draw[-] [draw=black,thick] (a44.center) -- (a55.center);
      
      \draw[-] [draw=black,thick] (a01.center) -- (a12.center);
      \draw[-] [draw=black,thick] (a12.center) -- (a23.center);
      \draw[-] [draw=black,thick] (a23.center) -- (a34.center);
      \draw[-] [draw=black, thick] (a34.center) -- (a45.center);

      \draw[-] [draw=black,thick] (a02.center) -- (a13.center);
      \draw[-] [draw=black,thick] (a13.center) -- (a24.center);
      \draw[-] [draw=black,thick] (a24.center) -- (a35.center);
      
      \draw[-] [draw=black,thick] (a03.center) -- (a14.center);
      \draw[-] [draw=black,thick] (a14.center) -- (a25.center);
      
      \draw[-] [draw=black,thick] (a04.center) -- (a15.center);
      \draw[-] [draw=black, thick] (a00.center) -- (a01.center);
      \draw[-] [draw=black,thick] (a01.center) -- (a02.center);
      \draw[-] [draw=black,thick] (a02.center) -- (a03.center);
      \draw[-] [draw=black,thick] (a03.center) -- (a04.center);
      \draw[-] [draw=black,thick] (a04.center) -- (a05.center);
      
      \draw[-] [draw=black,thick] (a11.center) -- (a12.center);
      \draw[-] [draw=black,thick] (a12.center) -- (a13.center);
      \draw[-] [draw=black,thick] (a13.center) -- (a14.center);
      \draw[-] [draw=black,thick] (a14.center) -- (a15.center);

      \draw[-] [draw=black,thick] (a22.center) -- (a23.center);
      \draw[-] [draw=black,thick] (a23.center) -- (a24.center);
      \draw[-] [draw=black,thick] (a24.center) -- (a25.center);
      
      \draw[-] [draw=black,thick] (a33.center) -- (a34.center);
      \draw[-] [draw=black,thick] (a34.center) -- (a35.center);
      
      \draw[-] [draw=black,thick] (a44.center) -- (a45.center);
      \draw[-] [draw=black,thick] (a01.center) -- (a11.center);
      
      \draw[-] [draw=black,thick] (a02.center) -- (a12.center);
      \draw[-] [draw=black,thick] (a12.center) -- (a22.center);
      
      \draw[-] [draw=black,thick] (a03.center) -- (a13.center);
      \draw[-] [draw=black, thick] (a13.center) -- (a23.center);
      \draw[-] [draw=black,thick] (a23.center) -- (a33.center);

      \draw[-] [draw=black,thick] (a04.center) -- (a14.center);
      \draw[-] [draw=black,thick] (a14.center) -- (a24.center);
      \draw[-] [draw=black,thick] (a24.center) -- (a34.center);
      \draw[-] [draw=black, thick] (a34.center) -- (a44.center);
      
      \draw[-] [draw=black,thick] (a05.center) -- (a15.center);
      \draw[-] [draw=black,thick] (a15.center) -- (a25.center);
      \draw[-] [draw=black,thick] (a25.center) -- (a35.center);
      \draw[-] [draw=black,thick] (a35.center) -- (a45.center);
      \draw[-] [draw=black,thick] (a45.center) -- (a55.center);
    \end{tikzpicture}
  \end{center}
  \caption{(a) The big hive triangle $\bighive$ for $n=5$, with the three kinds of rhombi marked. (b) The border labels of hives in $\Hive(\lmn)$ as functions of $\lmn$.}
  \label{fig:bighive-rhombi}
  \end{figure}

\begin{definition}\label{def:contentR}
Let $h$ be a $\reals$-labelling of the vertices of $\bighive$. Given a rhombus $R$ in $\bighive$, we define the {\em content} $R(h)$ of $R$ in $h$ to be the sum of the labels of $h$ on the obtuse-angled vertices of $R$ minus the sum of its labels on the acute-angled vertices of $R$. 
\end{definition}

    For example, if the vertex labels of the following rhombus $R$ are $a,b,c,d$ as shown, then it has content $(b+d) - (a+c)$:
    \begin{center}
        \begin{tikzpicture}
          \draw[thick](0,0)-- (60:1.5cm);
          \draw[thick](0.75,1.3) --(2.25,1.3);
          \draw[thick](0,0)-- (1.5,0);
          \draw[thick](1.5,0)-- (2.25,1.3);
          \node at (1.1,0.70) {$R$};
          \node at (-0.3,0) {$a$};
          \node at (0.73,1.5) {$b$};
          \node at (2.25,1.5) {$c$};
          \node at (1.7,0) {$d$};
        \end{tikzpicture}
    \end{center}

\begin{definition}\label{def:hivecone}
    The {\em hive cone} $\Hive$ is the set comprising all $\reals$-labellings $h$ of the vertices of $\bighive$ such that $R(h) \geq 0$ for each rhombus $R$ in $\bighive$.
    \end{definition}
It is clear that $\Hive$ is a polyhedral cone in $\reals^{(n+1)(n+2)/2}$, carved out by the rhombus inequalities of Definition~\ref{def:hivecone}. An element $h \in \Hive$ will be called a {\em hive}.  If $R(h)=0$, we say that the rhombus $R$ is {\em flat} in $h$ (see Figure~\ref{fig:example-map}). 

\begin{definition}
  Given $\lambda = (\lambda_1, \lambda_2, \cdots, \lambda_n) \in \reals^n$, let $|\lambda| = \sum_i \lambda_i$. Define the $(n+1)$-tuple of partial sums
  \[\ps[\lambda] = (0, \lambda_1, \sum_{i=1}^2 \lambda_i, \cdots, \sum_{i=1}^n \lambda_i)\]
  and the $(n-1)$-tuple of successive differences
  \[\sd \lambda = (\lambda_2 - \lambda_1, \lambda_3 - \lambda_2, \cdots, \lambda_n - \lambda_{n-1})\]
  so that $\sd \ps[\lambda] = \lambda$.
\end{definition}

\begin{definition}\label{def:hivepolytope}
    For $(\lambda, \mu, \nu) \in (\reals^n)^3$ with $|\lambda| + |\mu| = |\nu|$, the {\em hive polytope} $\Hive(\lmn)$ is the set of all hives whose boundary labels are $\ps[\lambda]$ (left edge, top to bottom), $|\lambda|+\ps[\mu]$ (bottom edge, left to right) and $\ps[\nu]$ (right edge, top to bottom) (figure~\ref{fig:bighive-rhombi}b).
\end{definition}    

For later use, we will find it convenient to label the NE rhombi of the big hive triangle $\bighive$. The NE rhombi in the last row are labelled $R_{nj}$ with $1 \leq j <n$ from left to right; the NE rhombi in the preceding row are labelled $R_{n-1,j}$ with $1 \leq j < n-1$ from left to right, and so on, moving one row upwards each time. The sole NE rhombus in the top row is labelled $R_{21}$. Figure~\ref{fig:NE-labels} depicts this labelling for $n=5$.

\begin{figure}[h]
\begin{center}
  \begin{tikzpicture}[scale=0.9]

  \node at (0.75, 0.25*1.732) {$R_{51}$};
  \node at (1.75, 0.25*1.732) {$R_{52}$};
  \node at (2.75, 0.25*1.732) {$R_{53}$};
  \node at (3.75, 0.25*1.732) {$R_{54}$};
  \node at (1.25, 0.755*1.732) {$R_{41}$};
  \node at (2.25, 0.75*1.732) {$R_{42}$};
  \node at (3.25, 0.75*1.732) {$R_{43}$};
  \node at (1.75, 1.25*1.732) {$R_{31}$};
  \node at (2.75, 1.25*1.732) {$R_{32}$};
  \node at (2.25, 1.75*1.732) {$R_{21}$};

      \node (a00) at (0,0*1.732) {};
      \node (a01) at (0.5,0.5*1.732) {};
      \node (a02) at (1,1*1.732) {};
      \node (a03) at (1.5,1.5*1.732) {};
      \node (a04) at (2,2*1.732) {};
      \node (a05) at (2.5,2.5*1.732) {};
      \node (a11) at (1,0*1.732) {};
      \node (a12) at (1.5,0.5*1.732) {};
      \node (a13) at (2,1*1.732) {};
      \node (a14) at (2.5,1.5*1.732) {};
      \node (a15) at (3,2*1.732) {};
      \node (a22) at (2,0*1.732) {};
      \node (a23) at (2.5,0.5*1.732) {};
      \node (a24) at (3,1*1.732) {};
      \node (a25) at (3.5,1.5*1.732) {};
      \node (a33) at (3,0*1.732) {};
      \node (a34) at (3.5,0.5*1.732) {};
      \node (a35) at (4,1*1.732) {};
      \node (a44) at (4,0*1.732) {};
      \node (a45) at (4.5,0.5*1.732) {};
      \node (a55) at (5,0*1.732) {};

      \draw[-] [draw=black,thick] (a00.center) -- (a11.center);
      \draw[-] [draw=black,thick] (a11.center) -- (a22.center);
      \draw[-] [draw=black,thick] (a22.center) -- (a33.center);
      \draw[-] [draw=black,thick] (a33.center) -- (a44.center);
      \draw[-] [draw=black,thick] (a44.center) -- (a55.center);
      
      \draw[-] [draw=black,thick] (a01.center) -- (a12.center);
      \draw[-] [draw=black,thick] (a12.center) -- (a23.center);
      \draw[-] [draw=black,thick] (a23.center) -- (a34.center);
      \draw[-] [draw=black, thick] (a34.center) -- (a45.center);

      \draw[-] [draw=black,thick] (a02.center) -- (a13.center);
      \draw[-] [draw=black,thick] (a13.center) -- (a24.center);
      \draw[-] [draw=black,thick] (a24.center) -- (a35.center);
      
      \draw[-] [draw=black,thick] (a03.center) -- (a14.center);
      \draw[-] [draw=black,thick] (a14.center) -- (a25.center);
      
      \draw[-] [draw=black,thick] (a04.center) -- (a15.center);
      \draw[-] [draw=black, thick] (a00.center) -- (a01.center);
      \draw[-] [draw=black,thick] (a01.center) -- (a02.center);
      \draw[-] [draw=black,thick] (a02.center) -- (a03.center);
      \draw[-] [draw=black,thick] (a03.center) -- (a04.center);
      \draw[-] [draw=black,thick] (a04.center) -- (a05.center);
      
      \draw[-] [draw=black,thick] (a11.center) -- (a12.center);
      \draw[-] [draw=black,thick] (a12.center) -- (a13.center);
      \draw[-] [draw=black,thick] (a13.center) -- (a14.center);
      \draw[-] [draw=black,thick] (a14.center) -- (a15.center);

      \draw[-] [draw=black,thick] (a22.center) -- (a23.center);
      \draw[-] [draw=black,thick] (a23.center) -- (a24.center);
      \draw[-] [draw=black,thick] (a24.center) -- (a25.center);
      
      \draw[-] [draw=black,thick] (a33.center) -- (a34.center);
      \draw[-] [draw=black,thick] (a34.center) -- (a35.center);
      
      \draw[-] [draw=black,thick] (a44.center) -- (a45.center);

       \draw[-] [draw=black,thick] (a05.center) -- (a15.center);
       \draw[-] [draw=black,thick] (a15.center) -- (a25.center);
       \draw[-] [draw=black,thick] (a25.center) -- (a35.center);
       \draw[-] [draw=black,thick] (a35.center) -- (a45.center);
       \draw[-] [draw=black,thick] (a45.center) -- (a55.center);
  \end{tikzpicture}
\end{center}
\caption{Labelling of North-East slanted rhombi, shown for $n=5$.}
\label{fig:NE-labels}
\end{figure}

\subsection{Hives and Gelfand--Tsetlin patterns}
We now fix $\lambda,\mu, \nu \in \pn$ with  $|\lambda| + |\mu| = |\nu|$. Since each $h \in \Hive(\lmn)$ is an $\reals$-labelled triangular array (of size $n+1$), its horizontal sections (marked in blue in figure~\ref{fig:example-map}) form a sequence of vectors $h_0, h_1, \cdots, h_n$ (listed from top to bottom), with $h_i \in \reals^{i+1}$. For example, for the hive on the left in figure~\ref{fig:example-map}, we have 
\begin{align*}
    h_0=(&0) \\ h_1 = (20&,\,28)\\ h_2 = (29,\, &45,\, 48)\\ h_3=(33,\,49&,\,57,\,57)\\ h_4= (35,\,51,\,&60,\,61,\,61)\\ h_5=(35,\, 51,\, 6&1,\, 63,\, 64,\, 64).
    \end{align*}
    We will often represent a hive $h$ by the tuple $(h_0, h_1, \ldots, h_n)$. 

Given $h = (h_0, h_1, \ldots, h_n) \in \Hive(\lmn)$, consider the tuple $(\sd h_1, \sd h_2, \cdots, \sd h_n)$.
Since $\sd h_i \in \reals^i$ for each $1 \leq i \leq n$, we may arrange this tuple into a triangular array of size $n$, with the entries of $\sd h_i$ arranged along the $i^{th}$ row from the top (see Figure \ref{fig:example-map}).

\begin{definition}
    The {\em horizontal successive differences} map $\bsd$ is given by:
\begin{equation}\label{eq:hsd}
  h=(h_0, h_1, \ldots, h_n) \mapsto \bsd h =  (\sd h_1, \sd h_2, \cdots, \sd h_n)
  \end{equation}
We view $\bsd$ as a linear map on the ambient spaces $\mathbb{R}^{\ell} \stackrel{\bsd}{\mapsto} \mathbb{R}^{\ell'}$, where $\ell = \frac{(n+1)(n+2)}{2}$ and $\ell' = \frac{n(n+1)}{2}$. 
\end{definition}
\begin{proposition}\label{prop:hivegt}
  For each $h \in \Hive(\lmn)$, we have $\bsd h \in \GT(\mu)$. 
\end{proposition}
\begin{proof}
  It is easy to see that the NE and SE rhombus inequalities satisfied by $h$ imply the corresponding NE and SE Gelfand--Tsetlin inequalities for $\bsd h$ (cf. \cite[Appendix]{buch}, \cite{King-terada-etal}). In other words, for all $n \geq i > j \geq 1$, we have
  \begin{equation}\label{eq:rnequal}
    R_{ij}(h) = \NE_{ij}(\bsd h)
    \end{equation}
  and an analogous statement for the SE rhombi.
\end{proof}
\begin{remark}
The chosen $\lambda$ and the vertical rhombus inequalities satisfied by $h$ do not play a role in the above proposition. They will figure in Proposition~\ref{prop:bsd-props}(3). 
\end{remark}
\begin{figure}[h]
\begin{center}
\begin{tikzpicture}[x={(1cm*0.5,-\rootthree cm*0.5)},y={(1cm*0.5,\rootthree cm*0.5)}]

  \node(a00) at (0,0){};
  \node(a03) at (0,3){};
  \node(a11) at (1,1){};
  \node(a14) at (1,4){};
  \fill[yellow,opacity=0.5]
  (a00.center) -- (a03.center) -- (a14.center) -- (a11.center) -- cycle;

\draw(0,0)node(a00){$35$};\draw(0,1)node(a01){$35$};\draw(0,2)node(a02){$33$};\draw(0,3)node(a03){$29$};\draw(0,4)node(a04){$20$};\draw(0,5)node(a05){$0$};
\draw(1,1)node(a11){$51$};\draw(1,2)node(a12){$51$};\draw(1,3)node(a13){$49$};
\draw(1,4)node(a14){$45$};\draw(1,5)node(a15){$28$};							                          \draw(2,2)node(a22){$61$};\draw(2,3)node(a23){$60$};\draw(2,4)node(a24){$57$};
\draw(2,5)node(a25){$48$};                                                                          \draw(3,3)node(a33){$63$};\draw(3,4)node(a34){$61$};\draw(3,5)node(a35){$57$};
\draw(4,4)node(a44){$64$};\draw(4,5)node(a45){$61$};		                
\draw(5,5)node(a55){$64$};
\foreach\i/\ii in{0/-1,1/0,2/1,3/2,4/3,5/4}
 \foreach\j/\jj in{0/-1,1/0,2/1,3/2,4/3,5/4}{
  \ifnum\i<\j \draw[draw=red](a\i\jj)--(a\i\j); \fi
  \ifnum\i>\j\else\ifnum\i>0 \draw(a\ii\j)--(a\i\j); \draw[draw=blue, ultra thick](a\ii\jj)--(a\i\j); \fi\fi

   }
\end{tikzpicture}\hspace{1cm}
\begin{tikzpicture}[x={(1cm*0.5,-\rootthree cm*0.5)},y={(1cm*0.5,\rootthree cm*0.5)}]
  \draw(0,0)node(a00){$16$};\draw(0,1)node(a01){$16$};\draw(0,2)node(a02){$16$};\draw(0,3)node(a03){$16$};\draw(0,4)node(a04){$8$};
  \draw(1,1)node(a11){$10$};\draw(1,2)node(a12){$9$};\draw(1,3)node(a13){$8$};\draw(1,4)node(a14){$3$};
  \draw(2,2)node(a22){$2$};\draw(2,3)node(a23){$1$};\draw(2,4)node(a24){$0$};
  \draw(3,3)node(a33){$1$};\draw(3,4)node(a34){$0$};
  \draw(4,4)node(a44){$0$};

\end{tikzpicture}
\end{center}
\caption{The hive on the left maps under $\bsd$ to the GT pattern on the right. The flat NE rhombi are highlighted.}
\label{fig:example-map}
\end{figure}

\begin{proposition}\label{prop:bsd-props}

Let $\lambda,\mu,\nu \in \pn$ with  $|\lambda| + |\mu| = |\nu|$. Then
\begin{enumerate}
    \item  $\bsd: \Hive(\lmn) \to \GT(\mu)$ is an injective map. 
    
    \item Let $h \in \Hive(\lmn)$. Then $ h \in \Hivez(\lmn) \iff \bsd h \in \GTz(\mu) $. 
    
    \item $\bg \,\circ\, \bsd$ is a bijection between $\Hivez(\lmn)$ and $\Tab_\lambda^\nu(\mu)$.
    \end{enumerate}
\end{proposition}

\begin{proof}
  (1)  Let $h = (h_0,h_1,\ldots,h_n)$ and $h' = (h'_0,h'_1,\ldots,h'_n)$ be two elements in $\Hive(\lmn)$, with $h_i,h'_i \in \mathbb{R}^{i+1}$ for $0 \leq i \leq n$. Let $h_i = (h_{i1},h_{i2},\ldots,h_{i(i+1)})$ and $h'_i = (h'_{i1},h'_{i2},\ldots,h'_{i(i+1)})$. We have $h_{01} = h'_{01}  = 0$ , $h_{i1} = \sum_{j=1}^i \lambda_j = h'_{i1}$  for $1 \leq i \leq n$.

Suppose that $\bsd h = \bsd h'$ that is $\sd h_i = \sd h'_i$ for $1 \leq i \leq n$. We will prove $h_{ik} = h'_{ik}$ by induction on $k$, for each fixed $i$. For the base case we have $h_{i1} = h'_{i1}$. Let $p \geq 2$ and suppose that $h_{ik} = h'_{ik}$ for $k < p$. Since $\sd h_{i} = \sd h'_{i}$, we have $h_{i(k+1)} - h_{ik} = h'_{i(k+1)} - h'_{ik}$ for all $1 \leq k \leq i$. Since $h_{ik} = h'_{ik}$ for $k<p$, this implies that $h_{ip} = h'_{ip}$. Hence we have $h_i = h'_i$ for $0 \leq i \leq n$.

\medskip(2) This is clear from the definition of $\bsd$ and the fact that the boundary labels of $h$ are integers.

\begin{figure}
\begin{center}
\begin{tikzpicture}[x={(1cm*0.5,-\rootthree cm*0.5)},y={(1cm*0.5,\rootthree cm*0.5)}]
  \draw(0,0)node(a00){$h_{n1}$};\draw(0,1)node(a01){$\cdot$};\draw(0,2)node(a03){$h_{21}$};\draw(0,3)node(a04){$h_{11}$};\draw(0,4)node(a05){$h_{01}$};
  \draw(1,1)node(a11){$h_{n2}$};\draw(1,2)node(a12){$\cdot$};\draw(1,3)node(a14){$h_{22}$};\draw(1,4)node(a15){$h_{12}$};
  \draw(2,2)node(a22){$\cdot$};\draw(2,3)node(a23){$\cdot$};\draw(2,4)node(a25){$h_{23}$};
  \draw(3,3)node(a33){$\cdot$};\draw(3,4)node(a34){$\cdot$};
  \draw(4,4)node(a55){$h_{n(n+1)}$};

\end{tikzpicture}\hspace{1cm}
\begin{tikzpicture}[x={(1cm*0.5,-\rootthree cm*0.5)},y={(1cm*0.5,\rootthree cm*0.5)}]
  \draw(0,0)node(a00){$a_{n1}$};\draw(0,1)node(a01){$\cdot$};\draw(0,2)node(a03){$a_{21}$};\draw(0,3)node(a04){$a_{11}$};
  \draw(1,1)node(a11){$a_{n2}$};\draw(1,2)node(a12){$\cdot$};\draw(1,3)node(a14){$a_{22}$};
  \draw(2,2)node(a22){$\cdot$};\draw(2,3)node(a23){$\cdot$};
  \draw(3,3)node(a33){$a_{nn}$};
\end{tikzpicture}
\end{center}
\caption{The left array is a hive $h$ and the right array is a GT pattern  $\bsd h = (a_{ij})$ such that $a_{ij} = h_{i(j+1)} - h_{ij}$ for $1 \leq j \leq i \leq n$.}
\label{fig:h-gth}
\end{figure}

\medskip(3) Let $h$ be an element of $\Hivez(\lmn)$. It follows from Proposition~\ref{prop:hivegt} that $T = \bg (\bsd h)$ is in $\Tab(\mu)$. We want to show that  $T \in \Tab_\lambda^\nu(\mu)$. Consider $h = (h_0,h_1,\ldots,h_n)$ where $h_i = (h_{i1},h_{i2},\ldots,h_{i(i+1)})$; then $\bsd h = (a_{ij})$ for $1 \leq j \leq i \leq n$ where $a_{ij} = h_{i (j+1)} - h_{ij}$ (see~figure~\ref{fig:h-gth}). From \S~\ref{sec:gt}, it follows that the number of $\boxed{i}$ in row $j$ of $T$ is $a_{ij} - a_{(i-1)j}$ for $j \leq i$ (setting $a_{i-1,i}=0$).

We want to show that $\supstd[\lambda]*b_T$ is a dominant word. Let $b_T =b_{T_1}*b_{T_2}*\cdots*b_{T_n}$ where $b_{T_k}$ is the reverse row reading word of the $k^{th}$-\text{row} of $T$. For each $1 \leq i \leq n$ and $0 \leq k \leq i$, let $N_{i,k}$ denote the number of occurrences of $i$ in the word $w_k:=\supstd[\lambda]*b_{T_1}*b_{T_2}*\cdots*b_{T_k}$. It follows easily from definition~\ref{def:domword} that $\supstd[\lambda]*b_T$ is a dominant word if and only if $N_{i,k} \geq N_{i+1, k+1}$ for all $1 \leq i <n$ and $0 \leq k \leq i$.
We have
\begin{equation}\label{eq:i2}
    N_{i,k} = \lambda_i + (a_{i1} - a_{i-1,1})+(a_{i2} - a_{i-1,2})+\cdots+(a_{i,k} - a_{(i-1),k})
\end{equation}
Using the definition of $\bsd$ to rewrite this in terms of $h$, we get:
\begin{equation}\label{eq:i3}
  N_{i,k} = h_{ik}  - h_{(i-1),k}
\end{equation}
Thus,
\begin{equation}\label{eq:vrhom}
  N_{i,k} - N_{i+1,k+1} = (h_{ik}  - h_{(i-1),k}) - (h_{(i+1),(k+1)} - h_{i,(k+1)})
\end{equation}
which is $\geq 0$ by the corresponding vertical rhombus inequality in $h$.
This proves that $\supstd[\lambda]*b_T$ is a dominant word.

Next we compute the weight of $\supstd[\lambda]*b_T$. The number of times $i$ appears in the $\supstd[\lambda]*b_T$ is:
\[\lambda_i + (a_{i1}-a_{(i-1),1})+(a_{i,2}-a_{(i-1),2})+\cdots+(a_{i,(i-1)} - a_{(i-1),(i-1)})+a_{ii} \]
\[ = \lambda_i + (h_{i2}-h_{(i-1),2}) - (h_{i1}-h_{(i-1),1}) + \cdots+ h_{(i+1),i} - h_{ii} \]
\[ = \lambda_i + (h_{i2}-h_{(i-1),2}) - \lambda_i + \cdots+ h_{i,(i+1)} - h_{ii} = h_{i,(i+1)} - h_{(i-1),i} = \nu_i  \]
This shows that the weight of the word $\supstd[\lambda]*b_T$ is $\nu$. Thus $\bg \,\circ\, \bsd$ is an injective map from $\Hivez(\lmn)$ to $\Tab_\lambda^\nu(\mu)$.

The surjectivity follows from similar arguments.
Let $T \in \Tab_\lambda^\nu(\mu)$; then $\bg^{-1}(T) = (a_{ij}) \in \GT(\mu)$ where $a_{ij}$ are defined recursively by $a_{(i-1),i} = 0$ and for $i \geq j$
\[ a_{ij} = (\# \ \boxed{i} \ \text{ in the } j^{th } \text{-row of } T) + a_{(i-1)j}.\]
Now let $h= (h_0,h_1,\ldots,h_n)$ where $h_0=(0)$ and $h_i = (h_{i1},h_{i2},\ldots,h_{i,(i+1)})$ for $ 1 \leq i \leq n$ are defined recursively by
\[ h_{i1} = \sum_{j=1}^i \lambda_j \text{ and } h_{i,(j+1)} = a_{ij}+h_{ij} \text{ for } 1 \leq j \leq i.\]
As before, we consider $h = (h_0,h_1,h_2,\ldots,h_n)$ as vertex labels of the big hive triangle $\bighive$ (see figure \ref{fig:h-gth}). It follows from the definition of $h$ that $\bg \circ \bsd(h) = T$. It only remains to show that $h \in \Hivez(\lmn)$. 

The NE and SE rhombi inequalities in $h$ hold since they transform under $\bsd$ to the corresponding Gelfand--Tsetlin inequalities in $\bsd (h) = (a_{ij}) \in \GT(\mu)$. The vertical rhombus inequalities in $h$ follow from equation~\eqref{eq:vrhom} since $\supstd[\lambda]*b_T$ is a dominant word. That the boundary labels of $h$ are $\ps[\lambda]$ (left edge, top to bottom), $|\lambda|+\ps[\mu]$ (bottom edge, left to right) and $\ps[\nu]$ (right edge, top to bottom) is easily verified.
\end{proof}


Note that the last assertion of proposition \ref{prop:bsd-props} implies that $|\Hivez(\lmn)| = \lr$ (and is a variation of proofs in \cite{buch}, \cite{pak-vallejo}). 

\subsection{Hive Kogan faces}
As before,  fix $\lambda,\mu, \nu \in \pn$ with  $|\lambda| + |\mu| = |\nu|$ and consider $\bsd: \Hive(\lmn) \to \GT(\mu)$. 
Given $F \subseteq \{(i,j): n \geq i > j \geq 1\}$, recall that $\kogan(\mu,F)$ is the face of $\GT(\mu)$ satisfying
\[ \kogan(\mu,F) = \{ A \in \GT(\mu): \NE_{ij}(A)=0 \text{ for } (i,j) \in F\} \]
We define $\koganH(\lmn, F) :=  \bsd^{-1} \kogan(\mu,F)$. In light of \eqref{eq:rnequal}, we conclude that this inverse image is given by
\[ \koganH(\lmn, F) = \{ h \in \Hive(\lmn): R_{ij}(h)=0 \text{ for } (i,j) \in F\}. \]
In other words, this is the face of $\Hive(\lmn)$ comprising hives in which $R_{ij}$ is flat for all $(i,j) \in F$. 
\begin{definition}\label{def:hivekoganface}
    We call $\koganH(\lmn, F)$ a {\em hive Kogan face} say that it is {\em reduced} if $F$ is.
\end{definition}

For $w \in S_n$, define the union of hive Kogan faces:
\begin{equation}\label{eq:hivekogandef}
  \koganH(\lmn, w) := \bsd^{-1} \kogan(\mu, w) = \bigcup \koganH(\lmn, F)
  \end{equation}
the union being over reduced $F$ for which $\varpi(F) = w$ (see definition~\ref{def:gtkogan}).

We have the following additive and scaling properties that readily follow from the definitions:
\begin{lemma}\label{lem:semiscale} Let $k \in \reals_{>0}$, $\lambda_i, \mu_i, \nu_i \in \pn$, $i=1,2$ and $F \subseteq \{(i,j): n \geq i > j \geq 1\}$. Then
\begin{align}
\koganH(\lambda_1, \mu_1, \nu_1, F) + \koganH(\lambda_2, \mu_2, \nu_2, F) &\subseteq \koganH(\lambda_1+\lambda_2, \mu_1+\mu_2, \nu_1+\nu_2, F) \label{eq:semigpprop}\\
  \koganH(k\lambda, k\mu, k\nu, F) &= k \, \koganH(\lmn, F) \notag\\
  \koganH(k\lambda, k\mu, k\nu, w) &= k \, \koganH(\lmn, w) \label{eq:scalingprop}
\end{align}
\end{lemma}
We now have the ingredients necessary to prove one of our main theorems, which gives
a hive description of the $\lrw$:
\begin{thm}\label{thm:hive-desc-lrw}
Let $\lmn \in \pn$ and $w \in S_n$. Then 
  $\lrw  = \# \koganHz(\lmn, \wnot w)$.
\end{thm}
\begin{proof}
  Theorem~\ref{thm:joseph} implies that 
  \[\lrw = |\demcrys_\lambda^\nu(\mu, w)| = |\Tab_\lambda^\nu(\mu) \cap \crys_w(\mu)|.\]
  Putting together Propositions~\ref{prop:fujita}, \ref{prop:bsd-props}, we have maps:
  \[ \Hivez(\lmn) \stackrel{\bsd}{\hookrightarrow} \GTz(\mu) \stackrel{\bg}{\rightarrow} \Tab(\mu) \]
  with $\bg\,\circ\,\bsd$ defining a bijection from $\Hivez(\lmn)$ to $\Tab_\lambda^\nu(\mu)$. Thus, the inverse image
  \[ (\bg \,\circ\,\bsd)^{-1} \left(\Tab_\lambda^\nu(\mu) \cap \crys_w(\mu)\right) = (\bg \,\circ\,\bsd)^{-1} \left(\crys_w(\mu)\right) = \bsd^{-1} \bg^{-1} \crys_w(\mu) \]
  has cardinality $\lrw$. But Proposition~\ref{prop:fujita} implies that $\bg^{-1} \crys_w(\mu)  = \koganz(\mu,w_0w)$ and $\bsd^{-1} \koganz(\mu,w_0w) = \koganHz(\lmn,w_0w)$ by definition and Proposition~\ref{prop:bsd-props}(2). 
\end{proof}

Putting together equation~\eqref{eq:scalingprop} and Theorem~\ref{thm:hive-desc-lrw}, we obtain a strengthening of Corollary~\ref{cor:scaling} in type $A$ (cf. Remark~\ref{rem:scaling}):
\begin{corollary}\label{cor:corscaling}
Let $\lmn \in \pn$ and $w \in S_n$. Then $c_{k\lambda, k\mu}^{k\nu}(w) \geq c_{\lambda \mu}^{\nu}(w)$ for all $k \geq 1$. In particular, $c_{\lambda \mu}^{\nu}(w) > 0$ implies that $c_{k\lambda, k\mu}^{k\nu}(w) > 0$ for all $k \geq 1$.  
\end{corollary}
We note that when $w$ is of special form, Theorem~\ref{thm:mainthm-312}(2) makes a stronger assertion than Corollary~\ref{cor:corscaling}, namely that the semigroup property holds for $w$ (see \S\ref{sec:sgsubsec} below).

\section{Pattern-avoiding permutations}\label{sec:patt312}
\subsection{} We recall the classical notion of pattern-avoidance.
\begin{definition}\label{def:pattavoid}
  Let $1 \leq k \leq n$. Let $\sigma$, $\tau$ be permutations of $\{1,2,\ldots,n\}$ and $\{1,2,\ldots,k\}$ respectively. We say that $\sigma$ {\em contains} $\tau$ if there exist $1 \leq i_1 < i_2 < \cdots < i_k \leq n$ such that for all distinct pairs $p, q \in \{1, 2, \ldots, k\}$, we have $\sigma(i_p) < \sigma(i_q)$ iff $\tau(p) < \tau(q)$.  We say that $\sigma$ is a
  $\tau${\em-avoiding permutation} if $\sigma$ does not contain $\tau$.
\end{definition}

It is easy to see from this definition that a permutation $w \in S_n$ is $312$-avoiding if there do not exist $1\leq i<j<k \leq n$ such that $w(j) < w(k) < w(i)$. Likewise, $w$ is $231$-avoiding if there do not exist $1\leq i<j<k \leq n$ such that $w(k) < w(i) < w(j)$ .

Our primary focus will be on $312$- or $231$-avoiding permutations. We observe that $w$ is $231$-avoiding if and only if $w^{-1}$ is $312$-avoiding. We also recall from Proposition~\ref{prop:properties-clmn} that for any $w$, we have 
\begin{equation}\label{eq:symm}
  \lrw[w^{-1}] = \lrevw
\end{equation}

\subsection{Saturation and semigroup theorem for type $A$}
Let $\lie g = \mathfrak{sl}_n$ for $n \geq 2$. Its Weyl group $W=W(\lie g)$ is isomorphic to $S_n$. We recall that a {\em proper Young subgroup} of $S_n$ is a subgroup of the form $H=S_{n_1} \times S_{n_2} \times \cdots \times S_{n_r}$ with $r \geq 2$ and $\sum_{i=1}^r n_i = n$, comprising permutations of $1, 2, \ldots, n$ that leave the first $n_1$ numbers invariant, the next $n_2$ numbers invariant and so on. Given $w \in H$, we write $w=(w_1, w_2, \ldots, w_r)$ and call the $w_i \in S_{n_i}$ the {\em components of} $w$. Notice that the $w_i$ commute pairwise. These notions may equivalently be formulated in the setting of Section~\ref{sec:wcomps}.

\begin{definition}\label{def:splform}
    A permutation $w \in S_n$ is said to be of {\em special form} if (a) it is $312$-avoiding, or (b) $231$-avoiding, or (c) belongs to a proper Young subgroup and each of its components $w_i$ is either $312$-avoiding or $231$-avoiding.
\end{definition}

The following theorem is one of the main results of this paper:
\begin{thm}\label{thm:mainthm-312}
    Let $\lie g = \mathfrak{sl}_n$ for $n \geq 2$. Let $w \in W(\lie g) = S_n$ be a permutation of special form. Then 
    \begin{enumerate}
        \item $w$ has the saturation property.
        \item $w$ has the strong semigroup (and hence the semigroup) property.
    \end{enumerate}
\end{thm}
Theorem~\ref{thm:mainthm-312}(2) will be proved below in Section~\ref{sec:sgsubsec}. 
We observe that $\wnot$ is $312$-avoiding (and also $231$-avoiding). Theorem~\ref{thm:mainthm-312}(1) thus generalizes the Knutson--Tao saturation theorem. The proof of Theorem~\ref{thm:mainthm-312}(1) uses ideas of Knutson-Tao, and occupies Sections \ref{sec:inchives} and \ref{sec:hivehorncones}.

We remark that we do not know if the saturation and semigroup properties hold for general $w \in W$, though we have been unable to find any counterexamples so far through our computer explorations on Sagemath.

\begin{remark}
    Permutations of {\em special form} occur in work of Postnikov-Stanley \cite[Theorem 13.4, Corollary 13.5, Remark 15.5]{PS}.
\end{remark}
\subsection{Hive Kogan faces and $312$-avoiding permutations}\label{sec:fw-312}
Let $w \in S_n$ be $312$-avoiding. Then, it is easy to see that $\wnot w$ is a $132$-avoiding permutation. We quote the following well-known characterization of such permutations (see for instance \cite[Proposition 2.2.1]{kogan} and \cite{PS,smirnov}).
\begin{proposition}\label{prop:132av}
  Let $\sigma \in S_n$. Then $\sigma$ is $132$-avoiding if and only if there exists a unique reduced $F_\sigma \subset \{(i,j): n \geq i > j \geq 1\}$ such that $\varpi(F_\sigma) = \sigma$.
  Further, there exists a sequence $1 \leq b_1 \leq b_2 \leq \cdots \leq b_{n-1} \leq n$ with $b_j \geq j$ for all $j$ such that $F_\sigma$ has the following form:
  \begin{equation}\label{eq:formfsigma}
    F_\sigma = \{(i,j): 1 \leq j < n, \; b_j < i \leq n\} 
  \end{equation}
  \end{proposition}
\begin{figure}
\begin{center}
    \begin{tikzpicture}[scale=0.9]
      \node (a00) at (0,0*1.732) {};
      \node (a01) at (0.5,0.5*1.732) {};
      \node (a02) at (1,1*1.732) {};
      \node (a03) at (1.5,1.5*1.732) {};
      \node (a04) at (2,2*1.732) {};
      \node (a05) at (2.5,2.5*1.732) {};
      \node (a11) at (1,0*1.732) {};
      \node (a12) at (1.5,0.5*1.732) {};
      \node (a13) at (2,1*1.732) {};
      \node (a14) at (2.5,1.5*1.732) {};
      \node (a15) at (3,2*1.732) {};
      \node (a22) at (2,0*1.732) {};
      \node (a23) at (2.5,0.5*1.732) {};
      \node (a24) at (3,1*1.732) {};
      \node (a25) at (3.5,1.5*1.732) {};
      \node (a33) at (3,0*1.732) {};
      \node (a34) at (3.5,0.5*1.732) {};
      \node (a35) at (4,1*1.732) {};
      \node (a44) at (4,0*1.732) {};
      \node (a45) at (4.5,0.5*1.732) {};
      \node (a55) at (5,0*1.732) {};

      \draw[-] [draw=black,thick] (a00.center) -- (a11.center);
      \draw[-] [draw=black,thick] (a11.center) -- (a22.center);
      \draw[-] [draw=black,thick] (a22.center) -- (a33.center);
      \draw[-] [draw=black,thick] (a33.center) -- (a44.center);
      \draw[-] [draw=black,thick] (a44.center) -- (a55.center);
      
      \draw[-] [draw=black,thick] (a01.center) -- (a12.center);
      \draw[-] [draw=black,thick] (a12.center) -- (a23.center);
      \draw[-] [draw=black,thick] (a23.center) -- (a34.center);
      \draw[-] [draw=black, thick] (a34.center) -- (a45.center);

      \draw[-] [draw=black,thick] (a02.center) -- (a13.center);
      \draw[-] [draw=black,thick] (a13.center) -- (a24.center);
      \draw[-] [draw=black,thick] (a24.center) -- (a35.center);
      
      \draw[-] [draw=black,thick] (a03.center) -- (a14.center);
      \draw[-] [draw=black,thick] (a14.center) -- (a25.center);
      
      \draw[-] [draw=black,thick] (a04.center) -- (a15.center);
      \draw[-] [draw=black, thick] (a00.center) -- (a01.center);
      \draw[-] [draw=black,thick] (a01.center) -- (a02.center);
      \draw[-] [draw=black,thick] (a02.center) -- (a03.center);
      \draw[-] [draw=black,thick] (a03.center) -- (a04.center);
      \draw[-] [draw=black,thick] (a04.center) -- (a05.center);
      
      \draw[-] [draw=black,thick] (a11.center) -- (a12.center);
      \draw[-] [draw=black,thick] (a12.center) -- (a13.center);
      \draw[-] [draw=black,thick] (a13.center) -- (a14.center);
      \draw[-] [draw=black,thick] (a14.center) -- (a15.center);

      \draw[-] [draw=black,thick] (a22.center) -- (a23.center);
      \draw[-] [draw=black,thick] (a23.center) -- (a24.center);
      \draw[-] [draw=black,thick] (a24.center) -- (a25.center);
      
      \draw[-] [draw=black,thick] (a33.center) -- (a34.center);
      \draw[-] [draw=black,thick] (a34.center) -- (a35.center);
      
      \draw[-] [draw=black,thick] (a44.center) -- (a45.center);

       \draw[-] [draw=black,thick] (a05.center) -- (a15.center);
       \draw[-] [draw=black,thick] (a15.center) -- (a25.center);
       \draw[-] [draw=black,thick] (a25.center) -- (a35.center);
       \draw[-] [draw=black,thick] (a35.center) -- (a45.center);
       \draw[-] [draw=black,thick] (a45.center) -- (a55.center);

       \draw [thick, draw=black, fill=yellow, fill opacity=0.5]
       (a00.center) -- (a03.center) -- (a14.center) -- (a13.center) -- (a24.center) -- (a23.center)-- (a45.center) -- (a44.center) -- cycle;
     \end{tikzpicture}
  \end{center}
\caption{A typical configuration of rhombi in $F_\sigma$ for $\sigma$ a $132$-avoiding permutation.}  
\label{fig:rhombi-rij}
\end{figure}
Pictorially, the union of the rhombi $R_{ij}$, $(i,j) \in F_\sigma$ forms a left-and-bottom justified region in the big hive triangle $\bighive$ (figure \ref{fig:rhombi-rij}). For $\sigma$ as in the above proposition, it follows from \eqref{eq:koganwdef} that $\kogan(\mu,\sigma) = \kogan(\mu, F_\sigma)$ for all $\mu \in \pn$.

We also note in passing the following description of the $b_i$ occurring in Proposition~\ref{prop:132av} (see \cite{PS}):
\begin{lemma}\label{lem:bidesc}
    For $\sigma \in S_n$ a $132$-avoiding permutation, the $b_i$ occurring in Proposition~\ref{prop:132av} are given by $$b_i = n - \left|\{j>i : \sigma(j) < \sigma(i)\}\right|. \qed$$
\end{lemma}
We now have the following polytopal descriptions of the $\lrw$ for $312$- and $231$-avoiding permutations:
\begin{proposition}\label{prop:312-231}
Let $w \in S_n$ be a $312$-avoiding permutation. Then,
\[ \lrw = \# \koganHz(\lmn,F_{w_0 w}) \]
If $w$ is 231-avoiding, then
\[ \lrw = \# \koganHz(\mln,F_{w_0 w^{-1}}) \]
\end{proposition}
\begin{proof}
As already remarked, the second assertion follows from the first by replacing $w$ by $w^{-1}$ and using Proposition~\ref{prop:properties-clmn}(d).  The first assertion follows from Theorem~\ref{thm:hive-desc-lrw}, Proposition~\ref{prop:132av} and the fact that $w_0 w$ is $132$-avoiding if $w$ is $312$-avoiding.
\end{proof}

\subsection{Proof of Theorem~\ref{thm:mainthm-312}(2): the strong semigroup property}\label{sec:sgsubsec}
We observe that Proposition~\ref{prop:312-231} and equation~\eqref{eq:semigpprop} together establish  Theorem~\ref{thm:mainthm-312}(2) when $w$ is $312$- or $231$-avoiding. The remaining case - when $w$ is a commuting product of such permutations in a Young subgroup - follows from Proposition~\ref{prop:w1w2cor}.

\subsection{Steinberg multiplicity formula} 
In this section, we derive an analogue of the Steinberg tensor product multiplicity formula for the $\lrw$ when $w$ is of special form.
Given $w$ $312$-avoiding, let $\sigma = \wnot w$ (this is $132$-avoiding). Let $b_i$ be the sequence corresponding to $\sigma$ as in Proposition~\ref{prop:132av} and Lemma~\ref{lem:bidesc}. Following \cite{PS}, we define the following subsets of the Weyl group (i.e., $S_n$) and the positive roots:
\begin{align*}
    W_w &= \{u \in S_n: u(i) \leq b_i \text{ for } 1 \leq i \leq n\} \\
    \Phi^+_{u,w} &= \{\varepsilon_i - \varepsilon_j: 1 \leq i < j \leq b_{u^{-1}(i)}\} \subseteq \Phi^+
\end{align*}
Let $\mu \in \pn$. We have the following character formula for the Demazure module $V_w(\mu)$ \cite{PS}:
\begin{equation}\label{eq:demchar312}
    \ch V_w(\mu) = \sum_{u \in W_w} \sgn(u) \frac{e^{u(\mu+\rho)-\rho}}{\prod_{\alpha \in \Phi^+_{u,w}} (1-e^{-\alpha})}
\end{equation}
We define an analogue of the Kostant partition function:
\[ K_{u,w}(\beta) = \text{coefficient of } e^{-\beta} \text{ in }\prod_{\alpha \in \Phi^+_{u,w}} (1-e^{-\alpha})^{-1}\]
In other words, $K_{u,w}(\beta)$ counts the number of partitions of $\beta$ into a sum of roots from $\Phi^+_{u,w}$. 
\begin{thm}\label{thm:postan312}
Let $w$ be $312$-avoiding and $\lmn \in \pn$. Then
\[ \lrw = \sum_{y \in W} \sum_{u \in W_w} \sgn(uy) K_{u,w}\left( \lambda + u(\mu+\rho) - y(\nu+\rho)\right)\]
\end{thm}
We omit the proof of this theorem, since it follows verbatim the proof of the classical Steinberg multiplicity formula (see for example \cite{Humphreys-book}). In place of the Weyl character formula for $\ch V(\mu)$ which appears in the latter, we use equation~\ref{eq:demchar312}.

We can obtain an analogous formula for the $231$-avoiding case via the relation $\lrw = c_{\mu\lambda}^{\nu}(w^{-1})$. When $w$ is a commuting product of such elements in a Young subgroup, we appeal to Proposition~\ref{prop:w1w2mult} to obtain a formula as a product.

\section{Increasable subsets for hives}\label{sec:inchives}
\subsection{}
As a first step toward the proof of our saturation theorem, we 
single out certain Kogan faces of the hive polytope that are closed under the operation of increasing the hive labels. 

More precisely, let $\lmn \in \reals^n$ with $|\lambda| + |\mu| = |\nu|$ and let $h \in \Hive(\lmn)$. Given a subset $S$ of the interior vertices of $\bighive$, consider the labelling $I_S$ (the indicator function) of the vertices of $\bighive$ which assigns the label $1$ to vertices of $S$ and $0$ to the remaining. 
\begin{definition}\label{def:incsub}
A subset $S$ of the interior vertices of $\bighive$ is said to be {\em increasable} for $h \in \Hive(\lmn)$ if there exists $\epsilon > 0$ such that $h' = h + \epsilon I_S \in \Hive(\lmn)$.
\end{definition}
  This notion is one of the central ideas of Knutson--Tao's proof of the saturation conjecture in the hive formalism. We quote the following important result \cite{KTHives,buch}:
  \begin{proposition}\label{prop:kt-increasable} (Knutson--Tao) Let $\lmn \in \reals^n$ be regular (i.e., $\lambda_i \neq \lambda_j$ if $i \neq j$, and likewise for $\mu, \nu$) with $|\lambda| + |\mu| = |\nu|$. Let $h$ satisfy the following properties:
\begin{enumerate}
\item $h$ is a vertex of the hive polytope $\Hive(\lmn)$,
\item $h$ has no increasable subsets.
\end{enumerate}
Then each interior label of $h$ is an integral linear combination of its boundary labels. In particular, if $\lmn \in \pn$, then $h \in \Hivez(\lmn)$.
\end{proposition}

\subsection{Increasable subsets for hives and $132$-avoidance}
Let $\lmn \in \reals^n$ with $|\lambda| + |\mu| = |\nu|$. The following simple observation is a crucial step in extending the Knutson--Tao method to our setting.

\begin{lemma}\label{lem:fw-increasable}
  Let $\sigma \in S_n$ be $132$-avoiding  and let $h \in \koganH(\lmn, F_\sigma)$. Let $S$ be an increasable subset for $h$, say $h' = h + \epsilon I_S \in \Hive(\lmn)$ for some $\epsilon > 0$. Then
  $h' \in \koganH(\lmn, F_\sigma)$.   
\end{lemma}
\begin{proof}
  We need to establish that $R_{ij}(h') =0$ for all $(i,j) \in F_\sigma$. Since $R(h') = R(h)$ for any rhombus $R$ of $\bighive$ whose vertices are disjoint from $S$, we will be done if we show that $S$ is disjoint from the set of vertices of the rhombi $R_{ij}$ for $(i,j) \in F_\sigma$.  This is trivial if $F_\sigma$ is empty. If $F_\sigma$ is non-empty, then by \eqref{eq:formfsigma}, we conclude that $(n,1) \in F_\sigma$. The rhombus  $R_{n1}$ has three vertices on the boundary (see Figure~\ref{fig:NE-labels}), and these cannot be in $S$. The fourth vertex is acute-angled, and if it belongs to $S$, then  $R_{n1}(h') < 0$, a contradiction. Moving on to the next rhombus $R_{n2}$ (if $(n,2) \in F_\sigma$), again three of its vertices cannot be in $S$ since they are either on the boundary or shared with $R_{n1}$. Neither can its fourth vertex, since it is acute-angled as before. Proceeding in this fashion, left-to-right along the rows, from the bottom row to the top, we conclude that none of the vertices of the $R_{ij}$ can belong to $S$ for $(i,j) \in F_\sigma$.
\end{proof}

\section{Proof of Theorem~\ref{thm:mainthm-312}(1)}\label{sec:hivehorncones}

\subsection{The Hive and Horn cones}

We recall the hive cone $\Hive$ from Definition~\ref{def:hivecone}. Consider the projection map
\[ \pi: \Hive \to \reals^{3n}, \;\; h \mapsto (\lmn) \]
where the boundary labels of $h$ are $\ps[\lambda], |\lambda| + \ps[\mu], \ps[\nu]$ as in Figure~\ref{fig:bighive-rhombi}. In other words, let $\alpha, \beta, \gamma$ respectively denote the left, bottom and right boundaries of $h$ (read top$\rightarrow$bottom, left$\rightarrow$right and top$\rightarrow$bottom respectively), then $\pi(h)  = (\sd \alpha, \sd \beta, \sd \gamma)$. 

The image
\begin{equation}\label{eq:horncone}
  \horn :=\pi(\Hive)
\end{equation}
is a polyhedral cone in $\reals^{3n}$. This coincides with the cone of spectra of triples $(A,B,C)$ of $n \times n$ Hermitian matrices with $C = A+B$ \cite{KTW}.

\subsection{Refined Hive and Horn cones}
\begin{definition}\label{def:sighivecone}
Let $\sigma \in S_n$ be $132$-avoiding and let $F_\sigma$ be as in \eqref{eq:formfsigma}.
The $\sigma$-{\em Hive cone}  is the set 
\[ \kcone[\sigma] := \{ h \in \Hive: R_{ij}(h) =0 \text{ for all } (i,j) \in F_\sigma\} \]
  \end{definition}

This is a face of $\Hive$ and clearly forms a polyhedral cone in its own right\footnote{A similar definition also works for all $\sigma \in S_n$, but in light of \eqref{eq:hivekogandef} and Proposition~\ref{prop:132av}, that would be a union of cones in general.}. By analogy to \eqref{eq:horncone}, we define:
\begin{equation}
\horn(\sigma) = \pi \left(\kcone[\sigma]\right)
\end{equation}
The scaling property \eqref{eq:scalingprop} shows that $\horn(\sigma)$ is closed under scaling by $k \in \reals_{>0}$. In fact, $\horn(\sigma)$ is a polyhedral cone in $\reals^{3n}$, being the image of a polyhedral cone under a projection map \cite[Lecture 1]{Ziegler}. In this notation, the Horn cone \eqref{eq:horncone} is $\horn = \horn(\id)$ where $\id \in S_n$ is the identity permutation.

\subsection{Proof of Theorem~\ref{thm:mainthm-312}(1) for $w$ $312$-avoiding}
With these definitions and Lemma~\ref{lem:fw-increasable} in place, we can use Knutson--Tao's arguments to prove Theorem~\ref{thm:mainthm-312}(1) for $w$ $312$-avoiding. Recall that if $w$ is $312$-avoiding, then $w_0 w$ is $132$-avoiding. Now by Theorem~\ref{thm:hive-desc-lrw}, $c_{k\lambda, k\mu}^{\,k\nu}(w) >0$ implies in particular that \[\Hive(k\lambda, k\mu, k\nu,\wnot w) \neq \emptyset.\] Thus $(k\lambda, k\mu, k\nu) \in \horn(\wnot w)$, and scaling by $k^{-1}$, we obtain that $(\lmn) \in \horn(\wnot w)$; in fact $(\lmn) \in \horn_\integers(\wnot w)$ since they were partitions to begin with.
So it is enough to establish the following statement (for $w$ $312$-avoiding):
\begin{equation}\label{eq:equiv-satprop}
    \koganHz(\lmn,\wnot w) \text{ is non-empty for all } (\lmn) \in \horn_\integers(\wnot w).
\end{equation}

To this end, we define the {\em largest lift map}, following \cite{KTHives, buch}. Choose a functional $\zeta$ on the cone $\kcone[\wnot w]$ which maps each hive $h$ to a generic positive linear combination of its interior vertex labels. Then, for each $(\lmn) \in \horn(\wnot w)$, the maximum value of $\zeta$ on $\pi^{-1}(\lmn)$ is attained  at a unique point; this point will be called its {\em largest lift}. The map
\[\ell: \horn(\wnot w) \to \kcone[\wnot w], \;\;\;\; (\lmn) \mapsto \text{ largest lift of } (\lmn) \]
is {continuous} and {piecewise-linear} \cite{buch,sturmfels-thomas}.

It is also clear that $\ell(\lmn)$ is a vertex of $\koganH(\lmn,\wnot w)$, thereby satisfying the first condition of Proposition~\ref{prop:kt-increasable}. We claim that it also satisfies the second condition there, i.e., that $h = \ell(\lmn)$ has no increasable subsets. For if $S$ is an increasable subset, let  $h' = h + \epsilon I_S \in \Hive(\lmn)$ for some $\epsilon > 0$. By Lemma~\ref{lem:fw-increasable}, $h' \in \koganH(\lmn,\wnot w)$. But $\zeta(h') > \zeta(h)$, violating maximality of $\zeta(h)$.

So Proposition~\ref{prop:kt-increasable} implies that for $\lmn$ regular, each interior label of $\ell(\lmn)$ is an integer linear combination of the $\lambda_i, \mu_i, \nu_i, \, 1 \leq i \leq n$.
As in \cite[\S 4]{buch} and \cite{KTHives}, by the continuity of $\ell$, it follows that each piece of $\ell$ is a linear function of $(\lmn) \in \reals^{3n}$ with $\integers$-coefficients. As a corollary:
\[ \ell(\horn_\integers(\wnot w)) \subseteq \kzcone[\wnot w] \]
Thus, for $(\lmn) \in \horn_\integers(\wnot w)$, we have $\ell(\lmn) \in  \koganHz(\lmn,\wnot w)$, thereby establishing \eqref{eq:equiv-satprop}.
This proves Theorem~\ref{thm:mainthm-312}(1) for $w$ $312$-avoiding. \qed

\subsection{Proof of Theorem~\ref{thm:mainthm-312}(1) for $231$-avoiding permutations}\label{sec:blocksproof}
If $w$ is $231$-avoiding, then $w^{-1}$ is $312$-avoiding. Proposition~\ref{prop:properties-clmn}(d) finishes the argument in this case (see also Proposition~\ref{prop:312-231}).

\subsection{Proof of Theorem~\ref{thm:mainthm-312}(1) for commuting products in Young subgroups} 
We note that Proposition~\ref{prop:w1w2mult} concludes the proof.

\section{A strengthened version of the saturation theorem}\label{sec:strongthm}
\subsection{}\label{sec:generalcase}  For $n=4$, the only permutations in $S_4$ which are not of the form of Theorem~\ref{thm:mainthm-312} are $3412, 3142, 2413, 4231$ (in one-line notation). For each of these $w$, there exist two reduced faces $F$ such that $\varpi(F) = w$. In these cases, $\kcone$ is a union of two polyhedral cones.

While our methods do not apply to a general $w \in S_n$ (beyond those covered by Theorem~\ref{thm:mainthm-312}), we do not know if the saturation property fails  there. In particular, a preliminary search using {\em Sage} for  $n=4, 5$ and small $\lmn, k$ did not turn up any counterexamples. 

\subsection{} We however have a strengthened version of the main theorem which recovers saturation for general $w$ at the cost of imposing restrictions on the $\lambda, \mu$.

For the following definition, recall here the notion of permutations of {\em special form} from Definition~\ref{def:splform}:
\begin{definition}
    Given a subset $I$ of $\{1, 2, \ldots, n-1\}$, let $W_I$ denote the parabolic subgroup of $S_n$ generated by $\{s_i: i \in I\}$. A pair  $(I, J)$ of such subsets will be called a {\em special pair} if every double coset in $W_I \backslash S_n /W_J$ has a representative of special form.
\end{definition}

We have the following simple properties of special pairs.
\begin{proposition}\label{prop:specpairprop}
Let $(I,J)$ be a special pair. Then:
\begin{enumerate}
    \item $(J,I)$ is a special pair.
    \item If $I' \supseteq I$ and $J' \supseteq J$, then $(I', J')$ is a special pair. \qed
\end{enumerate}
\end{proposition}
\begin{proof}
    The first assertion follows from the observation that $w$ is of special form iff $w^{-1}$ is of special form. The second is obvious.
\end{proof}
We will need the following modified version of Definition~\ref{def:sat}.
\begin{definition}\label{def:sat-new}
Fix $\lambda, \mu \in \pn$. A permutation $w \in S_n$ is said to have the {\em saturation property for $(\lambda, \mu)$} if the following holds for all $\nu \in \pn$: 
\begin{equation} \label{eq:satprop-new}
\lrwk > 0 \text{ for some } k\geq 1 \text{ implies } \lrw > 0.
\end{equation}
\end{definition}

\begin{thm}\label{thm:IJ}
Let $(I,J)$ be a special pair. Then every $w \in S_n$ has the saturation property for pairs of dominant weights $(\lambda,\mu)$ satisfying
\[ \lambda(\alpha_i^\vee) =0, \; \mu(\alpha_j^\vee)=0 \;\;\text{ for all } i \in I, j \in J.\]
\end{thm}
\begin{proof}
    For $\lambda, \mu$ as above, we have the stabilizers $W_\lambda \supseteq W_I$ and $W_\mu \supseteq W_J$. Since $(I,J)$ is a special pair, every double coset in $W_I \backslash S_n /W_J$ has a special form representative, and the same thereby holds for $W_\lambda \backslash S_n /W_\mu$. Given any $w \in S_n$, let $\sigma \in W_\lambda w W_\mu$ be of special form. Since $W_{k\lambda} = W_\lambda$ and $W_{k\mu} = W_\mu$ for all $k \geq 1$, we obtain $\lrwk = \lrwk[\sigma]$ by  Proposition~\ref{prop:kkmod}(1). The required conclusion now follows from Theorem~\ref{thm:mainthm-312}.
\end{proof}

In view of Proposition~\ref{prop:specpairprop}, we consider the partial order $\preceq$ on special pairs $(I,J)$ defined by:
\[ (I,J) \preceq (I',J') \text{ if } I \subseteq I' \text{ and } J \subseteq J'\]
The minimal special pairs with respect to this partial order would impose the least restrictions on $(\lambda,\mu)$ in Theorem~\ref{thm:IJ}. We note that one can appropriately modify  Theorem~\ref{thm:IJ} for the semigroup property as well. 
\begin{example}
    For $S_4$, the minimal special pairs are $(I,J) = (\emptyset, \{1\}), (\emptyset, \{3\}), (\{1\},\emptyset),(\{3\},\emptyset)$ as can be verified using the Remark in \S\ref{sec:generalcase}. Thus, every $w \in S_4$ has the saturation property for pairs $(\lambda,\mu)$ for which $\lambda(\alpha_i^\vee)=0$ or $\mu(\alpha_i^\vee)=0$ for $i=1$ or $i=3$.
\end{example}

\section{Symmetry of the $\lrw$}\label{sec:right-keys}
The symmetry $\lr = c_{\mu\lambda}^\nu$  was first studied via hives in \cite{Henriques-Kamnitzer}. There is another point-of-view stemming from Proposition~\ref{prop:bsd-props}, which leads to a bijective proof of the  general symmetry property
$$ \lrw = c_{\mu\lambda}^{\nu}(w^{-1})$$

First we recall some definitions and notations.
Consider the ``North-Easterly'' version $\bsd^{\scriptscriptstyle{NE}}$ of $\bsd$, which takes successive differences of labels along the $NE-SW$ direction (red edges of Figure~\ref{fig:example-map}) (see \cite[Example 2.8]{King-terada-etal} and \cite[Appendix]{buch}, whose hive-drawing conventions differ from ours and from each other!). Consider $h \in \Hive(\lmn)$ then $\bsd^{NE}(h)$ is a GT pattern of shape $\lambda$, which can be interpreted as a {\em contretableau} $T^\dag$ of shape $\lambda$ \cite{buch}. The map $\bsd^{NE}$ is injective, as follows by arguments similar to those establishing Proposition \ref{prop:bsd-props}.

Fix a subset $F \subseteq \{(i,j): n \geq i > j \geq 1\}$. Consider the face of $\GT(\mu)$ obtained by setting $a_{i-1,j} - a_{i,j+1}=0$ for $(i,j) \in F$ and leaving all other inequalities untouched.  
We call this the {\em dual Kogan face} $\overline{\kogan}(\mu, F)$. To each pair $(i,j) \in \{(i,j): n \geq i > j \geq 1\}$, associate the simple transposition $s_{j} \in S_n$. We consider the total order on pairs $(i,j)$ defined by $(i,j)$ precedes $(i',j') \iff$ either $i<i'$, or $i=i'$ and $j>j'$. We list the elements of $F$ in increasing order relative to this total order. Denote the product of the corresponding $s_{j}$ in this order by $\overline{\sigma}(F)$. If $\len \overline{\sigma}(F) = |F|$, i.e., this word is reduced, we say that $F$ is {\em reduced} and set \cite[Definition 5.1]{fujita}:
\[ \overline{\varpi}(F) = \wnot \,\overline{\sigma}(F) \,\wnot \]
For $w \in S_n$, let $\overline{\kogan}(\mu, w) := \bigcup \overline{\kogan}(\mu, F)$, the union over reduced $F$ for which $\overline{\varpi}(F) = w$.
We can now state the following result of Fujita \cite[Corollary 5.2]{fujita}:

\begin{proposition}\label{prop:fujita1}
 There is a bijection between $\overline{\koganz}(\mu, \wnot w \wnot)$ and $ \demcrys(\mu, w)^{op}$. 
\end{proposition}

We also recall from Fujita \cite[\S 2]{fujita} that there is an involution $\eta_\mu:\Tab(\mu) \to \Tab(\mu)$ such that:
\begin{equation}\label{eq:invdem}
    \eta_\mu(\demcrys(\mu, w)) = \demcrys(\mu, \wnot w)^{op}; \hspace{1cm} \eta_\mu(\demcrys(\mu, w)^{op}) = \demcrys(\mu, \wnot w).
\end{equation}
Putting together Proposition \ref{prop:fujita}, Proposition \ref{prop:fujita1} and equation \eqref{eq:invdem}, we get the following:
\begin{equation}\label{eq:invkogan}
    \eta_\mu(\koganz(\mu, \wnot w )) = \overline{\koganz}(\mu,  w \wnot); \hspace{1cm} \eta_\mu(\overline{\koganz}(\mu, \wnot w \wnot)) = \koganz(\mu, w )
\end{equation}

\begin{lemma}\label{lem:wwinv}
    Let $h \in \koganHz(\lmn, \wnot w)$ then $\bsd h \in \koganz(\mu, \wnot w )$ and $\bsd^{NE}h \in \overline{\koganz}(\lambda, w^{-1} \wnot )$.
\end{lemma}
\begin{proof}
Clearly $\bsd h \in \koganz(\mu, \wnot w )$ by the definition of $\koganHz(\lmn, \wnot w)$. This means that $\bsd h \in \kogan(\mu,F) $ for some reduced face $F$ of $\GT(\mu)$ such that $\sigma(F) = w \wnot$. Fix such a reduced face $F$, then $h \in \koganHz(\lmn,F_0)$, where $F_0$ is the hive reduced face such that $R_{ij}$ is flat for all $(i,j) \in F$ . Observe that $\bsd^{NE}h \in \overline{\koganz}(\lambda, F_0)$, where $F_0$ is thought as a face of $\GT(\lambda)$.

By Remark \ref{rem:sigma} we know that $\sigma(F) = \sigma'(F)$. Observe that $\overline{\sigma}(F_0) = \sigma'^{-1}(F)$ since $\overline{\sigma}(F_0)$ is the product of $s_i$'s in the reverse order of the product of $s_i$'s in $\sigma'(F)$. Then we have $\overline{\sigma}(F_0) = \wnot w^{-1}$ and $\overline{\varpi}(F_0) = w^{-1} \wnot$. This shows that $\bsd^{NE}h \in \overline{\koganz}(\lambda, w^{-1} \wnot )$.
\end{proof}

We will now construct a bijective map $\Psi: \koganHz(\lmn, \wnot w) \to \koganHz(\mu,\lambda,\nu, \wnot w^{-1})$. Let $h \in \koganHz(\lmn, \wnot w)$. By Lemma \ref{lem:wwinv} we have $\bsd^{NE}h \in \overline{\koganz}(\lambda, w^{-1} \wnot )$. From equation \eqref{eq:invkogan}, $\eta_\lambda(\overline{\koganz}(\lambda, w^{-1} \wnot )) = \koganz(\lambda, \wnot w^{-1} )$, which implies that $\eta_\lambda(\bsd^{NE}h) \in \koganz(\lambda, \wnot w^{-1} )$.
   
   From \cite[Appendix A]{buch} we know that $\eta_\lambda(\bsd^{NE}h)$ is in $Tab_\mu^\nu(\lambda)$, which implies that $\eta_\lambda(\bsd^{NE}h)$ is in the image of the injective map $\bsd : \koganHz(\mu,\lambda,\nu, \wnot w^{-1}) \to \koganz(\lambda, \wnot w^{-1})$.  Denote the preimage of $\eta_\lambda(\bsd^{NE}h)$ under $\bsd$ by $h^*$, that is $h^* = \bsd^{-1}(\eta_\lambda(\bsd^{NE}h))$.
   
   Now we define $\Psi(h) = h^*$. Clearly the map $\Psi$ is a well defined injective map since $\Psi = \bsd^{-1} \circ \eta_\lambda \circ \bsd^{NE}$ is a composition of injective maps. The inverse map of $\Psi$ can be easily defined in a similar way such that $\Psi^{-1} =(\bsd^{NE})^{-1} \circ \eta_\lambda \circ \bsd$.

\printbibliography

\end{document}